\documentclass[reqno,final]{amsart}

\usepackage{caption}
\usepackage{subcaption}
\usepackage{amsfonts,latexsym,graphicx,amsmath,amssymb,bbm,enumitem,url}
\usepackage{algorithm}
\usepackage{algpseudocode}

\makeindex

\usepackage[color,notref,notcite]{showkeys}
\allowdisplaybreaks

\definecolor{labelkey}{rgb}{0.6,0,1}

\usepackage[normalem]{ulem}
\newcounter{corr}
\definecolor{violet}{rgb}{0.580,0.,0.827}
\newcommand{\corr}[3]{\typeout{Warning : a correction remains in page
\thepage}
				\stepcounter{corr}        
				{\color{blue}\ifmmode\text{\,\sout{\ensuremath{#1}}\,}\else\sout{#1}\fi}
        {\color{red}#2}
        {\color{violet} \fbox{\thecorr}#3}}

\newcounter{cst}
\catcode`\@=11
\def\ctel#1{C_{\refstepcounter{cst}\@bsphack
\protected@write\@auxout{}%
           {\string\newlabel{#1}{{\thecst}{\thepage}}}\thecst}}
\catcode`\@=12

\newcounter{cexp}
\catcode`\@=11
\def\terml#1{T_{\refstepcounter{cexp}\@bsphack
\protected@write\@auxout{}%
           {\string\newlabel{#1}{{\thecexp}{\thepage}}}\thecexp}}
\catcode`\@=12

\newcommand{\mathbi}[1]{{\boldsymbol #1}}

\newcommand{\eop}{{\unskip\nobreak\hfil\penalty50
           \hskip2em\hbox{}\nobreak\hfil\mbox{\rule{1ex}{1ex} \qquad}
   \parfillskip=0pt
   \finalhyphendemerits=0\par\medskip}}
\renewenvironment{proof}[1][]{\noindent {\bf Proof#1. } }{\eop}

\newtheorem{theorem}{Theorem}[section]
\newtheorem{remark}[theorem]{Remark}

\newtheorem{lemma}[theorem]{Lemma}

\definecolor{shadecolor}{gray}{0.92}
\definecolor{TFFrameColor}{gray}{0.92}
\definecolor{TFTitleColor}{rgb}{0,0,0}

\newcommand{\ba}{\begin{array}{llll}   }
\newcommand{\bac}{\begin{array}{c}}
\newcommand{\bari}{\begin{array}{r}}
\newcommand{\ea}{\end{array}}

\newcommand{\ban}{\begin{array}{llll}}
\newcommand{\ean}{\end{array}}

\newcommand{\be}{\begin{equation}}
\newcommand{\ee}{\end{equation}}

\newcommand{\beqsys }{\beqtab \left \{ \begin{array}{l}}
\newcommand{\eeqsys }{\end{array} \right . \eeqtab }

\newcommand{\benum}{\begin{enumerate}}
\newcommand{\eenum}{\end{enumerate}}

\newcommand{\beqtab}{\begin{eqnarray}} 
\newcommand{\eeqtab}{\end{eqnarray}}

\newcommand{\bfn}{\mathbf{n}}

\newcommand{\centercv}{\x_\cv}                         

\newcommand{\cv}{K}

\newcommand{\dcvedge}{d_{\cv,\edge}}

\newcommand{\disc}{{\mathcal D}}

\newcommand{\edge}{\sigma}

\newcommand{\edges}{{\mathcal E}}              
\newcommand{\edgescv}{{{\edges}_\cv}}

\newcommand{\grad}{\nabla}

\newcommand{\Lam}{\mathbf{\Lambda}}	

\newcommand{\mesh}{{\mathcal M}}

\newcommand{\ncvedge }{\bfn_{\cv,\edge}}  

\newcommand{\norm}[2]{\| #1 \|_{#2}}

\renewcommand{\O}{\Omega}

\newcommand{\R}{\mathbb R}

\newcommand{\V}{\mathbf{V}}

\newcommand{\x}{\mathbi{x}}

\newcommand{\centeredge}{\overline{\mathbi{x}}_\edge} 

\renewcommand{\norm}[2]{\left\Vert#1\right\Vert_{#2}}

\def\ograd{\overline{\grad}}

\def\gradD{\nabla_\disc}

\usepackage{xparse}
\DeclareDocumentCommand{\RPiD}{ O{\disc} O{,0} }{\Pi_{#1}(X_{#1#2})}

\def\Fdof#1{{\bm{\mathcal{F}}(#1,\R)}}
\makeatletter
\def\Fdof{\@ifnextchar[{\@with}{\@without}}
\def\@with[#1]#2{{\bm{\mathcal{F}}(#2;#1)}}
\def\@without#1{{\bm{\mathcal{F}}(#1,\R)}}
\makeatother

\def\RT0{\mathbb{RT}_0}

\def \hfillx {\hspace*{-\textwidth} \hfill}
\definecolor{labelkey}{rgb}{0.6,0,1}
\makeatletter
\newcommand*\bigcdot{\mathpalette\bigcdot@{.5}}
\newcommand*\bigcdot@[2]{\mathbin{\vcenter{\hbox{\scalebox{#2}{$\m@th#1\bullet$}}}}}
\def\BState{\State\hskip-\ALG@thistlm}
\makeatother

\begin{document}

	\title{A generalised complete flux scheme for anisotropic advection-diffusion equations}
	
		\author{Hanz Martin Cheng}
\address{Department of Mathematics and Computer Science, Eindhoven University of Technology, P.O. Box 513, 5600 MB Eindhoven, The Netherlands.
	\texttt{h.m.cheng@tue.nl}}
\author{Jan ten Thije Boonkkamp}
\address{Department of Mathematics and Computer Science, Eindhoven University of Technology, P.O. Box 513, 5600 MB Eindhoven, The Netherlands.
	\texttt{j.h.m.tenthijeboonkkamp@tue.nl}}
	
	\date{\today}
	
	%
	%
	%
	\maketitle
		\begin{abstract}
			In this paper, we consider separating the discretisation of the diffusive and advective fluxes in the complete flux scheme. This allows the combination of several discretisation methods for the homogeneous flux with the complete flux (CF) method. In particular, we explore  the combination of the hybrid mimetic mixed (HMM) method and the CF method, in order to utilize the advantages of each of these methods. The usage of HMM allows us to handle anisotropic diffusion tensors on generic polygonal (polytopal) grids, whereas the CF method provides a framework for the construction of a uniformly second-order method, even when the problem is advection dominated. 
		\end{abstract}
	\section{Introduction}
\noindent	Let $\O$ be an open connected subset of $\mathbb{R}^d$ $(d=2,3)$ with boundary $\partial\O = \Gamma_\mathrm{D} \cup \Gamma_\mathrm{N} $, where $\Gamma_\mathrm{D}$ and $\Gamma_\mathrm{N}$ are disjoint. We consider the advection-diffusion problem: Find $c\in H^1(\O)$ such that 
	\begin{equation}\label{eq:adv-diff}
	\begin{aligned}
	\nabla \bigcdot (c\V-\Lam\nabla c) &= s \mbox{ on } \Omega,\\
	c &= g \mbox{ on } \Gamma_\mathrm{D}, \\
	\Lam\nabla c \bigcdot \mathbf{n} &= h \mbox{ on } \Gamma_\mathrm{N}.
	\end{aligned}
	\end{equation}
	 Here,  we assume that $g,h \in L^2(\partial\O)$, the source term $s \in L^2(\O)$, the velocity profile $\V \in C^1(\O)^d$, and the diffusion tensor $\Lam$ is a 
	 symmetric positive definite field in $L^2(\O)^{d\times d}$.
	\begin{remark}[boundary condition]
If $\Gamma_\mathrm{D}\neq \emptyset$, then \eqref{eq:adv-diff} has a unique solution $c$. On the other hand, if $\Gamma_\mathrm{D} = \emptyset$, it can be shown that the kernel is one-dimensional and spanned by a function that has non-zero average \cite{DV09-NeumannBC}. Hence, for uniqueness of $c$, we impose in this case
	\begin{equation}\nonumber
	\frac{1}{|\O|}\int_\O c \,\mathrm{d}A= \bar{c},
	\end{equation}
	where $\bar{c}$ is a specified average value of $c$ in $\O$.
	\end{remark}

Stationary advection-diffusion problems \eqref{eq:adv-diff} are a fundamental component of non-stationary advection-diffusion equations of the form
\[
c_t+\nabla \bigcdot (c\V-\Lam\nabla c) = s \mbox{ on } \Omega \times (0,T),
\]
which are usually encountered in computational fluid dynamics. To be specific, these types of problems are encountered in applications to plasma physics \cite{L08-plasma} and porous media flow (e.g. reservoir engineering \cite{EW01-summary-advection-dominated,PR62} and groundwater flow \cite{ESZ09-ADG}). In particular, for flows in porous media, heterogeneous and highly anisotropic diffusion tensors are involved, due to varying rock properties (such as permeability and porosity) in the domain \cite{LM15-mathematicalModels}.

The uniformly second-order complete flux (CF) scheme \cite{AB11-FVCF} was originally developed on Cartesian grids for \eqref{eq:adv-diff} with scalar diffusion, i.e., $\Lam= D(\x) \mathbf{I}_d$. Some other uniformly second (or higher-order) schemes for the advection--diffusion problem \eqref{eq:adv-diff} with scalar diffusion, applicable on generic meshes, can be found in \cite{BM-2dFV,MR08-FV-conv-dominated}. On the other hand, some recent approaches to deal with anisotropic diffusion tensors while maintaining a high  order of accuracy involve the use of discontinuous skeletal or hybrid high order methods \cite{HHOBook20,DSGD15-adv-diff,DE15-HHO}, multipoint flux approximations (MPFA) \cite{YA08-DG-MPFA}, or the introduction of nonlinear fluxes \cite{L10-monotoneFV}. The aim of this paper is to extend the complete flux  schemes in \cite{AB11-FVCF} in order to handle anisotropic diffusion, whilst maintaining its uniformly second-order convergence properties. Since the complete flux scheme is linear, it is in general cheaper to implement than nonlinear methods. We should note however that compared to nonlinear methods, the complete flux scheme does not guarantee that the maximum principle is satisfied.

The generalised complete flux scheme we propose in this work only requires one unknown inside each cell, and another one on each interface. For advection-dominated problems, hybrid high order methods with polynomials of degree $k$ achieve an order of accuracy between $h^{k+1}$ and $h^{k+1/2}$, depending on the P\'eclet number \cite{DSGD15-adv-diff}. On the other hand, combining the ideas behind the complete flux scheme together with the hybrid high order scheme with polynomials of degree 0, which corresponds to the HMM \cite{dro-10-uni} or the SUSHI \cite{EGH10-SUSHI} scheme is an efficient way to achieve second-order accuracy without introducing additional unknowns.

The novelties of this work are the following:
\begin{itemize}
	\item Splitting the diffusive and advective components of the flux, which allows for a combination of different numerical discretisations with the CF method. In particular, we describe the combination of the hybrid mimetic mixed (HMM) and the CF method.
	\item A generalisation of the CF method on nonuniform meshes in one dimension, which paves a way to formulating the CF method on generic meshes in dimension $d>1$.
	\item The formulation of a generalised grid-based P\'eclet number, which allows us to use the CF method for highly anisotropic problems.
	\item An alternative derivation of the CF method in two-dimensional Cartesian meshes, which can be straightforwardly extended into three dimensions.
	\item The introduction of unknowns along the faces (edges), which allows us to have a more compact stencil for the CF method. In particular, the stencil will only involve the direct neighbors of a cell $K$, i.e., in the case of Figure \ref{fig:stencils}, introduction of edge unknowns will yield a stencil that involves only five cells, whereas without the edge unknowns, the stencil involves nine cells.
\end{itemize}

\begin{figure}[h!]
	\begin{tabular}{cc}
		\includegraphics[width=0.45\linewidth]{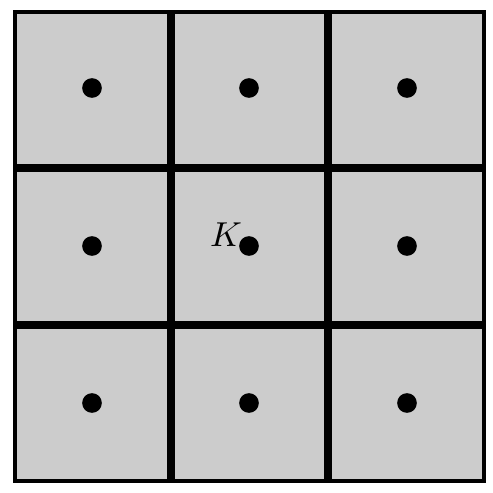} & \includegraphics[width=0.45\linewidth]{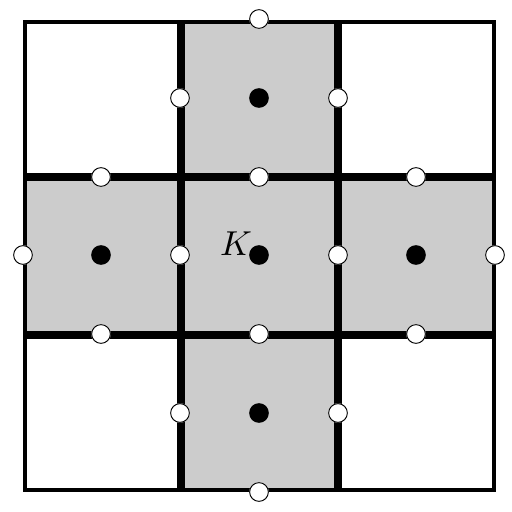} 
	\end{tabular}
	\caption{Stencils for the discretisation: cells involved in the stencil are shaded gray. Left(cell unknowns only); right(including edge unknowns). }
	\label{fig:stencils}
\end{figure}

The outline of the paper is as follows: In Section \ref{sec:HF}, we present a finite volume (homogeneous flux) discretisation for \eqref{eq:adv-diff}, and give details on how we discretise the diffusive and advective fluxes. Following this, we describe in Section \ref{sec:CF} the introduction of inhomogeneous fluxes, which are fundamental in the construction of the complete flux schemes. Numerical tests are then provided in Section \ref{sec:Numtests} to illustrate the second-order accuracy of the combined HMM--CF scheme. Finally, Section \ref{sec:Summary} provides a summary, as well as possible directions for future research.

\section{Finite volume methods for the advection-diffusion problem}\label{sec:HF}
For the discretisation, we start by defining a mesh as in \cite[Section 3.1]{DSGD15-adv-diff}, i.e., a partition of $\O$ into polyhedral (polygonal) cells. We then denote by $\mathcal T=(\mesh, \edges)$  the set of cells $K$ and faces (edges in 2D) $\edge$ of our mesh, respectively. The set of faces is then written as a disjoint union of interior faces $\edges_{\mathrm{int}}$ and exterior (boundary) faces $\edges_{\mathrm{ext}}$, $\mathcal{E} = \edges_{\mathrm{int}} \cup  \edges_{\mathrm{ext}}$, where $\edges_{\mathrm{int}}\cap \partial \O = \emptyset$.
For a cell $K\in\mesh$, we denote by $\edgescv\subset \edges$ the set of faces (edges) of the cell $K$. The numerical solution produced by our discretisation is piecewise constant, and hence, involves one unknown on each cell and  another one on each of the faces (edges), and we denote the space of unknowns by
\[
X_{\disc} := \{w=((w_{\cv})_{\cv\in\mesh},(w_{\edge})_{\edge\in\edgescv}) : w_K \in \mathbb{R}, w_\sigma \in \R  \}.
\] 

To write a finite volume scheme, we start by taking the integral of the first equation in \eqref{eq:adv-diff} over a control volume $K$, and use the divergence theorem to obtain the balance of fluxes:
\[ \sum_{\edge \in \edgescv} \int_\sigma (-\Lam\nabla c + c\V) \bigcdot \mathbf{n}_{K,\sigma} \,\mathrm{d}\ell = \int_{K} s \,\mathrm{d}A,
\]
where $\mathbf{n}_{K,\sigma}$ is the unit outer normal to the face $\sigma\in\edgescv$. Now, if $\sigma$ is a face shared by two distinct cells $K$ and $L$, then 
 \[\int_\sigma (-\Lam\nabla c + c\V) \bigcdot \mathbf{n}_{K,\sigma} \,\mathrm{d}\ell + \int_\sigma (-\Lam\nabla c + c\V) \bigcdot \mathbf{n}_{L,\sigma} \,\mathrm{d}\ell  = 0.\] This is known as the conservation of fluxes.
 
We now denote the approximations to the diffusive and advective fluxes by $F_{K,\sigma}^{D}$ and $F_{K,\sigma}^{A}$, respectively, i.e.
\[F_{K,\sigma}^{D} \approx \int_\sigma (-\Lam\nabla c)\bigcdot \mathbf{n}_{K,\sigma} \,\mathrm{d}\ell , \qquad  F_{K,\sigma}^{A}\approx \int_\sigma (c\V) \bigcdot \mathbf{n}_{K,\sigma} \,\mathrm{d}\ell .
\] Key to the formulation of a finite volume scheme is the definition of discrete fluxes (diffusive and advective)  so that the fluxes satisfy a discrete version of the balance and conservation of fluxes \cite{D14-FVschemes}. That is,
\begin{subequations}\label{eq:disc-fv}
for all $K\in\mesh$,

\begin{equation}\label{eq:fluxbal}
\sum_{\edge \in \edgescv} (F_{K,\sigma}^D + F_{K,\sigma}^A)  = s_K |K|,
\end{equation}
 and for each face $\sigma$ shared by distinct cells $K$ and $L$,
\begin{equation}\label{eq:fluxcons}
(F_{K,\sigma}^D+F_{K,\sigma}^A) + (F_{L,\sigma}^D+F_{L,\sigma}^A) = 0. 
\end{equation}
\end{subequations}
Here, $s_K$ is a piecewise constant approximation of the source term $s$ in cell $K$.
For this paper, we will use the HMM method \cite{dro-10-uni} for the diffusive fluxes, and the Scharfetter--Gummel (SG) method \cite{SG69} for the advective fluxes.
\subsection{HMM method for diffusion}
In this section, we discuss the discretisation of the diffusive fluxes $F_{K,\sigma}^D$ by using the HMM method. The choice of using the HMM method is due to the fact that it can handle anisotropic diffusion tensors. To construct the fluxes, we start by introducing a piecewise constant gradient, which is defined on a sub-triangulation of cells. Let $\x_K\in K$ be a point in cell $K$ such that $K$ is star-shaped with respect to $\x_K$ (a natural choice is taking $\x_K$ to be the cell barycenter), and for all $\sigma\in\edges_{K}$, define $d_{K,\sigma}>0$ to be the orthogonal distance between $\x_K$ and $\sigma$ (see Figure \ref{fig.subd}).  We then set $\centeredge$ to be the centre of mass of $\sigma$. In 2D, this is simply the midpoint of the edge $\sigma$.

\begin{figure}[h]
	\caption{Notations in a generic cell in dimension $d=2$.}
	\begin{center}
		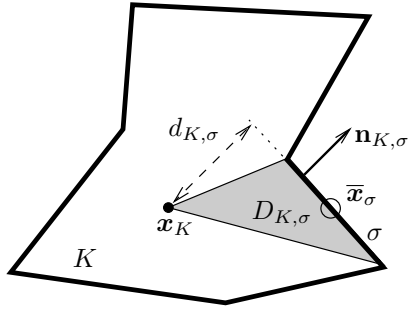
	\end{center}
	\label{fig.subd}
\end{figure} Following \cite{dro-10-uni}, if $K\in\mesh$ and $(D_{K,\edge})_{\edge\in\edgescv}$
is the convex hull of $\edge$ and $\x_K$ (see Figure \ref{fig.subd}), we set
\begin{equation}\label{def.grad} 
\begin{aligned}
&\forall w\in X_{\disc}\,,\;\forall \x\in D_{K,\edge}\,,\\
&\grad_{\disc} w = \ograd_{\cv} w + \dfrac{\sqrt{d}}{\dcvedge}[w_{\edge}-w_{\cv}-\ograd_{\cv}w \bigcdot(\centeredge-\centercv)] \ncvedge\,,\\
&\ograd_{\cv}w=\dfrac{1}{|\cv|}\sum_{\edge\in\edgescv} |\sigma|(w_{\edge}-w_K)\ncvedge.
\end{aligned}
\end{equation}
Here, $\ograd_\cv w$ is a linearly exact reconstruction of the gradient, i.e., if $(w_{\edge})_{\edge\in\edgescv}$ interpolates an affine function $a$ at the edge midpoints, then $\ograd_\cv w=\nabla a$.  The second term in $\grad_{\disc} w(\mathbf{x})$ is a stabilisation term, which ensures the coercivity of the numerical scheme. 

For the HMM method, the definition of the discrete diffusive flux is based on the bilinear form $a(u,w):=\int_\O \Lam \nabla u \bigcdot \nabla w\, \mathrm{d}A$, which stems from the weak formulation of the advection-diffusion equation \eqref{eq:adv-diff}. To be specific, 
the definition of the fluxes $(F_{K,\edge}^D) _{K\in\mesh,\,\edge\in\edgescv}$ is inspired by the relation
\begin{equation}\label{eq:insp_diff_flux}
\int_{K} \Lam \nabla u \bigcdot \nabla w \,\mathrm{d}A= \sum_{\edge \in \edgescv} \int_\sigma (\Lam \nabla u) \bigcdot \mathbf{n}_{K,\sigma} \gamma(w)\,\mathrm{d}\ell  - \int_{K} w \nabla \bigcdot(\Lam \nabla u) \,\mathrm{d}A,
\end{equation}
where $\gamma:H^1(K) \rightarrow L^2(\partial K)$ is the trace operator, with 
\[
\gamma(w) = (w)_{\left|\partial K\right.}  \qquad \forall w\in H^1(K).
\]
Under the assumption that $w$ is constant in cell $K$ with value $w_K$, we can write
\begin{equation}\nonumber
\begin{aligned}
\int_K w\nabla \bigcdot (\Lam \nabla u) \,\mathrm{d}A &= w_K \int_K \nabla \bigcdot (\Lam \nabla u) \,\mathrm{d}A \\ 
&= w_K \sum_{\sigma\in\edgescv} \int_\sigma (\Lam \nabla u) \bigcdot \mathbf{n}_{K,\sigma} \,\mathrm{d}\ell.
\end{aligned}
\end{equation}
 For $c\in X_{\disc}$, the fluxes $(F_{K,\edge}^D) _{K\in\mesh,\,\edge\in\edgescv}$ are then defined by a discrete counterpart of \eqref{eq:insp_diff_flux},
\begin{equation}\label{eq:diff_flux}
\begin{aligned}
&\forall K\in\mesh\,,\;\forall v \in X_{\disc}\,,\\
&\int_{K} \Lam \nabla_{\disc} c(\mathbf{x}) \bigcdot \nabla_{\disc} v(\mathbf{x}) \,\mathrm{d}A =: \sum_{\edge \in \edgescv} F_{K,\sigma}^D (v_{K}-v_{\sigma}).
\end{aligned}
\end{equation}
We note that the fluxes $F_{K,\edge}^D$ in \eqref{eq:diff_flux} are uniquely defined. In particular, we can see from \eqref{def.grad} that $\gradD v$ is uniquely determined by the values of $(v_\sigma-v_K)$. Hence, for a given edge $\sigma\in\edges_{K}$, $F_{K,\sigma}$ can be uniquely determined from \eqref{eq:diff_flux} by setting, for example, $v_{\sigma}-v_K=1$ and $v_{\sigma'}-v_K=0$ for all the other edges $\sigma'\in\edges_{K}$. 

\subsection{SG method for advection}
We now discuss in this section the discretisation of the advective fluxes $F_{K,\sigma}^A$. The original SG method gives a discretisation for both the diffusive and advective fluxes simultaneously. However, this original formulation of the SG method does not work in cases where diffusion is anisotropic. In this section, we use a modification of the SG method, which only gives the discretisation of the advective fluxes, introduced in \cite{VDM11-adv-diff}. Instead of using the Bernoulli function, we use
\[
A_{\mathrm{sg}}(t) = \dfrac{-t}{e^{-t}-1}-1.
\] 
Following the ideas in \cite{AB85-MFEM,VDM11-adv-diff}, we use a hybridised discretisation for the advective flux, using edge unknowns instead of unknowns from neighboring cells. The hybridised SG method for advective fluxes then reads: 
For each $K\in\mesh, \edge\in\edgescv$,
\begin{equation}\label{eq:advFlux}
\begin{aligned}
F_{K,\sigma}^A &= \frac{\lambda_\sigma|\sigma|}{d_{K,\sigma}}\bigg(A_{\mathrm{sg}}\bigg(\frac{d_{K,\sigma} V_{K,\sigma}}{\lambda_\sigma}\bigg) c_K - A_{\mathrm{sg}}\bigg(-\frac{d_{K,\sigma} V_{K,\sigma}}{\lambda_\sigma}\bigg) c_\sigma\bigg), \\
V_{K,\sigma} &= \frac{1}{|\sigma|}\int_\sigma \V \bigcdot \mathbf{n}_{K,\sigma} \,\mathrm{d}\ell.
\end{aligned}
\end{equation}

 The quantity $\lambda_\sigma$ is defined depending on whether $\sigma$ is an interior edge shared by cells $K$ and $L$, or $\sigma$ is a boundary edge. Here, we set 
\[
\lambda_\sigma = \begin{cases}
\min(1,\mbox{spec}(\Lam_K),\mbox{spec}(\Lam_L)) \quad \mbox{ if } \quad  \sigma \in \edgescv \cap \mathcal{E}_L,\\
\min(1,\mbox{spec}(\Lam_K)) \quad \mbox{ otherwise, }
\end{cases}
\]
where $\mbox{spec}(\Lam_K)$ are the eigenvalues of $\Lam_K$. The purpose of scaling $d_{K,\sigma} V_{K,\sigma}$ by the minimum eigenvalue of $\Lam_K$ is to bring enough numerical diffusion to ensure a better
stability for advection dominated problems \cite{VDM11-adv-diff}. In cases with strong anisotropic diffusion, however, this definition of $\lambda_\sigma$ might introduce excessive numerical diffusion. An improvement in the definition of $\lambda_\sigma$, which captures directional diffusion better, will then be introduced in equation \eqref{eq:genLambda} of Section \ref{sec:Pecletno}.

\subsection{Finite volume (homogeneous) fluxes}\label{subsec:HF}
In this section, we define the finite volume (homogeneous) fluxes, given by 
\begin{equation}\label{eq:HF}
F_{K,\sigma}^{H} := F_{K,\sigma}^D + F_{K,\sigma}^A.
\end{equation}

For the discrete flux balance equation \eqref{eq:fluxbal}, we need to compute both $\sum_{\edge \in \edgescv} F_{K,\sigma}^D$ and $\sum_{\edge\in\edgescv} F_{K,\sigma}^A$. The first sum is obtained by
taking $v\in X_{\disc}$ such that $v_K=1$ and 0 elsewhere in \eqref{eq:diff_flux}. The latter sum is obtained by taking the sum over $\edge\in\edgescv$ in \eqref{eq:advFlux}. Combining these, we impose the discrete flux balance equation \eqref{eq:fluxbal} 
\[
\sum_{\edge \in \edgescv} F_{K,\sigma}^H   = s_K|K|.
\]

Now, if $\edge\in\edgescv, K\in\mesh$ is an interior edge, we take $v\in X_{\disc}$ such that $v_\edge=1$, and 0 for all other components in \eqref{eq:diff_flux} in order to obtain $F_{K,\sigma}^D$. 
Combining this with the advective flux $F_{K,\sigma}^A$ defined in \eqref{eq:advFlux} then allows us to construct $F_{K,\sigma}^H$. We then add this to $F_{L,\sigma}^H$ in order to impose the conservativity of fluxes \eqref{eq:fluxcons}; that is,
\[
F_{K,\sigma}^H + F_{L,\sigma}^H = 0.
\] 

Finally, we describe the fluxes along the boundary edges. Given a cell $K\in\mesh$, if $\edge\in\edgescv \cap \Gamma_N$, that is, $\sigma$ is a Neumann boundary edge, then we still take $v\in X_{\disc}$ such that $v_\edge=1$, and 0 for all other components in \eqref{eq:diff_flux} to get $F_{K,\sigma}^D$. We then impose the Neumann boundary conditions by setting
	\[ 
	\frac{1}{|\sigma|} F_{K,\sigma}^D = -h.
	\]
	On the other hand, if $\sigma$ is a Dirichlet boundary edge, we impose
	\[
	c_\sigma = \frac{1}{|\sigma|} \int_\sigma g \,\mathrm{d}\ell.
	\]

A homogeneous flux scheme then constitutes of solving the above system of linear equations. At this stage, we note that other methods may be used for discretising the homogeneous diffusive and advective fluxes.  To maintain the second-order accuracy of the scheme, a natural choice for the discretisation is such that the approximations to both the diffusive and advective fluxes are second order.
\section{The complete flux scheme} \label{sec:CF}
In this section, we discuss the complete flux (CF) scheme. The main idea behind the CF scheme is the introduction of inhomogeneous fluxes $F_{K,\sigma}^I$, which results in an extension of the stencil onto neighboring elements. In particular, the discrete balance of fluxes is now given by 
\begin{equation}\label{eq:complete_fluxbal}
\sum_{\edge\in\edgescv} (F_{K,\sigma}^H + F_{K,\sigma}^I) =  s_K|K| ,
\end{equation}
where $F_{K,\sigma}^I$ are the inhomogeneous fluxes, which, among others, involve source terms coming from the cell $K$ and its neighboring elements. The inhomogeneous fluxes are defined so that for each cell $K\in\mesh$, they are nonzero only on interior edges $\sigma\in\edges_{K} \cap \edges_{\mathrm{int}}$. Inhomogeneous fluxes are not needed for imposing boundary conditions. Further details about the construction of the inhomogeneous fluxes will be discussed in Sections \ref{sec:CFnD} and \ref{sec:IFnd}.

\begin{remark}[Localised stencil]
	With the hybridised discretisation of the advective fluxes in \eqref{eq:advFlux}, the expression \eqref{eq:complete_fluxbal} gives a more localised stencil compared to the original formulation of the complete flux scheme \cite{AB11-FVCF} (see Figure \ref{fig:stencils}, left and right). 
\end{remark}
\subsection{Complete flux scheme in one dimension}\label{sec:CF1d}
To better understand the formulation of the complete flux scheme, we start by recalling its derivation in one dimension, following the ideas in \cite{AB11-FVCF}. In one dimension, \eqref{eq:adv-diff} reads: Find $c\in H^1(\O)$ such that 
\begin{equation}\label{eq:adv-diff-1D}
(-\Lambda c' + cV)' = s, \qquad x_0<x<x_N.
\end{equation}
 We assume that the source term $s\in L^2(\O)$.
Proper boundary conditions (Dirichlet or Neumann) are then imposed at $x_0$ and $x_N$. We then form a partition of the domain $x_0<x_1<\dots<x_{N-1}<x_N$. For $ x\in(x_j,x_{j+1})_{j=0,1,\dots,N-1}$, the idea behind the complete flux scheme is to solve a local boundary value problem
\begin{equation}\nonumber
\begin{aligned}
&f' = s, \qquad x_j<x<x_{j+1},\\
&c(x_j) = c_j, \qquad c(x_{j+1}) = c_{j+1},
\end{aligned}
\end{equation}
where the flux is defined as $f:=-\Lambda c' + cV$. Now, we define
\begin{equation}\nonumber
\hat{P}= \int_{x_{\sigma_j}}^x \frac{V}{\Lambda} \,\mathrm{d}x',  \qquad \hat{s} = \int_{x_{\sigma_j}}^x s \,\mathrm{d}x',
\end{equation}
where $x_{\sigma_j}$ is a point between $x_j$ and $x_{j+1}$ where the flux has to be evaluated. We also introduce a scaled coordinate $\alpha_{\sigma_j} \in (0,1)$ so that
\begin{equation}\label{eq:sigma_lin_comb_1D}
x_{\sigma_j} = (1-\alpha_{\sigma_j}) x_j + \alpha_{\sigma_j} x_{j+1}.
\end{equation}
It can then be shown that an analytical expression for $f$ is given by
\begin{equation}\label{eq:expFlux}
f = -\Lambda\big(ce^{-\hat{P}}\big)' e^{\hat{P}}, \qquad  x\in(x_j,x_{j+1}).
\end{equation}
Integrating \eqref{eq:adv-diff-1D} from $x_{\sigma_j}$ to $x\in (x_j,x_{j+1})$, we obtain
\[
f(x) - f(x_{\sigma_j})= \hat{s}.
\]
Substituting the explicit expression of the flux obtained in \eqref{eq:expFlux} into the equation above then gives
\[
f_{\sigma_j} =  -\Lambda\big(ce^{-\hat{P}}\big)' e^{\hat{P}} -\hat{s},
\]
where $ f_{\sigma_j} = f(x_{\sigma_j})$. Multiplying by $\Lambda^{-1}e^{-\hat{P}}$, we obtain the relation
\begin{equation}\label{eq:f_sig_1D}
(\Lambda^{-1}e^{-\hat{P}})  f_{\sigma_j} = -  \big(ce^{-\hat{P}}\big)' -(\Lambda^{-1}e^{-\hat{P}}) \hat{s}.
\end{equation}

 Integrating \eqref{eq:f_sig_1D} from $x_j$ to $x_{j+1}$ and dividing both sides of the expression by $\int_{x_j}^{x_{j+1}}(\Lambda^{-1}e^{-\hat{P}})\mathrm{d}x$ yields
\begin{equation}\label{eq:completeFlux_1D}
\begin{aligned}
f_{\sigma_j} &= f_{\sigma_j}^H + f_{\sigma_j}^I,\\
f_{\sigma_j}^H &= -\dfrac{e^{-\hat{P}_{j+1}}c_{j+1}-e^{-\hat{P}_j}c_j}{\int_{x_j}^{x_{j+1}}\Lambda^{-1}e^{-\hat{P}} \mathrm{d}x}, \\
f_{\sigma_j}^I &=  -\dfrac{\int_{x_j}^{x_{j+1}}\Lambda^{-1}e^{-\hat{P}}\hat{s}\,\mathrm{d}x}{\int_{x_j}^{x_{j+1}}\Lambda^{-1}e^{-\hat{P}} \mathrm{d}x},
\end{aligned}
\end{equation}
where $\hat{P}_j= \hat{P}(x_j)$.
\begin{remark}[Homogeneous flux]
	We note that taking $x_{\sigma_j} = x_{j+\frac{1}{2}}$, the one-dimensional homogeneous flux $f_{\sigma_j}^H$ in \eqref{eq:completeFlux_1D} corresponds to the original (second-order accurate) Scharfetter--Gummel flux \cite{SG69}, where diffusion and advection are discretised simultaneously. The approximation of the homogeneous flux in \eqref{eq:completeFlux_1D} can be replaced by other discretisations in diffusion and advection, e.g. combining a two-point flux discretisation for diffusion with a centered scheme for advection. A natural choice for maintaining the accuracy of the numerical scheme is to use second-order schemes for discretising both the diffusive and advective fluxes.
\end{remark}

 We now focus on the inhomogeneous flux. For simplicity of exposition, we assume only in this section that the diffusion and velocity field $\Lambda, V\in\mathbb{R}$ are constants. Using a piecewise constant approximation for the source term, an explicit expression for the inhomogeneous flux is then given by 
 \begin{equation}\label{eq:inhomog_flux_1D}
 f_{\sigma_j}^I = \Delta x_j \bigg(Z(-P_{\sigma_j},1-\alpha_{\sigma_j})s_j -Z(P_{\sigma_j},\alpha_{\sigma_j})s_{j+1}\bigg),
 \end{equation}
 where $s_j=s(x_j)$, $\Delta x_j = x_{j+1}-x_{j}$, and
 \begin{equation}\label{eq:func_Z}
 Z(P,\alpha) = \dfrac{e^{\alpha P}-1-\alpha P}{P(e^{P}-1)},
 \end{equation}
 and the local P\'eclet number $P_{\sigma_j}$ is defined to be
 \begin{equation}\label{eq:Peclet-1D}
 P_{\sigma_j} = \frac{V}{\Lambda}\Delta x_j.
 \end{equation}
\begin{remark}[P\'eclet number]
	We note that in one dimension, the definition of the (localised) P\'eclet number \eqref{eq:Peclet-1D} is quite straightforward, due to the fact that 
	\begin{itemize}
				\item No complication arises in taking the quotient in \eqref{eq:Peclet-1D}, since both $V$ and $\Lambda$ are scalars. 
				
			\item The direction of the velocity is only either towards the left or right, and hence the sign of the P\'eclet number is either negative or positive, respectively. This was automatically taken care of in \eqref{eq:inhomog_flux_1D}. 
	\end{itemize}

	In order to extend to dimension $d>1$, a generalised P\'eclet number needs to be introduced, to ensure that the direction of the velocity and the relative local strength of advection over diffusion are captured properly.
\end{remark}

The main difference in this formulation of the one-dimensional complete flux scheme on nonuniform grids compared to \cite{FL17-FVCF-unstructured-grids,ABK17-FVCF-2D} is the introduction of $\alpha_{\sigma_j}$ in \eqref{eq:sigma_lin_comb_1D}, which leads to the function $Z(P,\alpha)$ in \eqref{eq:func_Z}. In the literature, $\alpha_{\sigma_j} = 0.5$; however, in two dimensions or higher, it might occur, especially when the grid is distorted, that $\alpha_{\sigma_j} \neq 0.5$. In this case, the definition of $\alpha_{\sigma_j}$ and $Z(P,\alpha)$ as in \eqref{eq:sigma_lin_comb_1D} and \eqref{eq:func_Z} would be required.

\subsection{Generalised P\'eclet number in higher dimensions} \label{sec:Pecletno}
As seen in \eqref{eq:inhomog_flux_1D}, an expression for the P\'eclet number is needed for computing the inhomogeneous flux in one dimension. This is also the case for dimension $d>1$. Hence, we need to define an extension of the local P\'eclet number \eqref{eq:Peclet-1D}, so that it is applicable in higher dimensions. In this section, we present two ways of defining a local P\'eclet number $P_{K,\sigma}$ along the edge $\sigma$ shared by cells $K$ and $L$. We denote by $\Lam_K$ the average value of $\Lam$ in cell $K$ and by $\V_{\sigma}$ the average value of $\V$ on $\sigma$.

\subsubsection{Eigenvector-based P\'eclet number}
Firstly, we may define the local P\'eclet number as
\begin{equation}\label{eq:Peclet1_2D}
P_{K,\sigma} = |\x_K - \x_L| (\Lam_K^{-1}\V_{\sigma}) \bigcdot \mathbf{n}_{K,\sigma}.
\end{equation}
This is a straightforward modification of \eqref{eq:Peclet-1D}, where the quotient is obtained by taking the matrix inverse of $\Lam_K$, and the direction of the velocity field is taken into account by taking the dot product with the unit outer normal $\mathbf{n}_{K,\sigma}$.

However, this is highly dependent on the eigenvalues of $\Lam_K$. If the eigenvalues of $\Lam_K$ have high contrast, then in general, this would always yield a P\'eclet number $P_{K,\sigma}$ which is very large, regardless of the direction of $\mathbf{n}_{K,\sigma}$. To see this, we write an orthogonal diagonalisation $\Lam_K = \mathbf{U}_K\mathbf{D}_K\mathbf{U}_K^T$, where $\mathbf{D}_K$ is the diagonal matrix containing the eigenvalues of $\Lam_K$ and $\mathbf{U}_K$ is the orthogonal matrix containing the eigenvectors of $\Lam_K$. From this, we have
\begin{equation}\label{eq:Peclet1}
\begin{aligned}
(\Lam_K^{-1}\V_\sigma) \bigcdot \mathbf{n}_{K,\sigma } &= (\mathbf{D}_K^{-1}\mathbf{U}_K^T \V_\sigma) \bigcdot (\mathbf{U}_K^T\mathbf{n}_{K,\sigma}).
\end{aligned}
\end{equation}
Note that the eigenvectors of $\Lam_K$ form an orthonormal basis for $\mathbb{R}^d$. As a consequence, we can write
\begin{equation}\label{eq:onb}
\begin{aligned}
\V_\sigma &= \sum_{i=1}^d \gamma_i\mathbf{u}_{K,i},\\
\mathbf{n}_{K,\sigma} &= \sum_{i=1}^d \beta_i \mathbf{u}_{K,i},
\end{aligned}
\end{equation}
where $\mathbf{u}_{K,i}$ are the column vectors of $\mathbf{U}_K$ with 
\[
\mathbf{u}_{K,i}^T \mathbf{u}_{K,j} = \delta_{ij}.
\]

Denoting by $\lambda_{K,i}$ the eigenvalue of $\Lam_K$ that corresponds to the eigenvector $\mathbf{u}_{K,i}$, we then use \eqref{eq:Peclet1} and \eqref{eq:onb} to write 
\begin{equation} \label{eq:Peclet1_2D_exp} 
(\Lam_K^{-1}\V_\sigma) \bigcdot \mathbf{n}_{K,\sigma } = \sum_{i=1}^d \dfrac{\gamma_i}{\lambda_{K,i}}\beta_i.
\end{equation}
Hence, the term containing the smallest eigenvalue $\lambda_{K,j}$ of $\Lam_K$ will dominate (provided that $\gamma_j\beta_j$ is nonzero), especially if the condition number $\kappa(\Lam_K)$ of $\Lam_K$ is large. Moreover, we see in \eqref{eq:onb} and \eqref{eq:Peclet1_2D_exp} that the strength of advection over diffusion $(\frac{\gamma_i}{\lambda_{K,i}})$ is computed with respect to the basis $\{\mathbf{u}_{K,i}\}_{i=1,\dots,d}$. That is, $(\frac{\gamma_i}{\lambda_{K,i}})$ measures the strength of advection over diffusion in the direction of the eigenvector $\mathbf{u}_{K,i}$. Next, the P\'eclet number along $\mathbf{n}_{K,\sigma}$ is computed by a weighted average, where the weights are determined by writing $\mathbf{n}_{K,\sigma}$ as a linear combination of $\{\mathbf{u}_{K,i}\}_{i=1,\dots,d}$. We note that due to computing the strength of advection over diffusion in the direction of the eigenvector $\mathbf{u}_{K,i}$, a wrong sign may be obtained for the P\'eclet number along $\mathbf{n}_{K,\sigma}$.  

As an example, given a constant velocity field $\V=(1,2)^T$, and a constant diffusion tensor \begin{equation}\label{eq:Lam_ill}\Lam=\dfrac{1}{2} \begin{bmatrix}
1+10^{-8} & 1-10^{-8}  \\
1-10^{-8} & 1+10^{-8}
\end{bmatrix},\end{equation} then we would expect the P\'eclet number along the $x$- and $y$-directions to be moderate throughout the mesh. Considering a cell $K\in\mesh$, we take $\Lam_K=\Lam$ and $\V_\sigma=\V$ for all $\sigma\in\edges_{K}$. Now, we compute the condition number $\kappa(\Lam_K)=10^8$ of $\Lam_K$. This is due to the fact that the eigenvalues of $\Lam_K$, $\lambda_{K,1}=1$ and $\lambda_{K,2}=10^{-8}$, have very large contrast. Also, we observe that the corresponding eigenvectors are given by 
\[
\mathbf{u}_{K,1} = \big(\cos \frac{3\pi}{4}, -\sin \frac{3\pi}{4}\big)^T, \quad
\mathbf{u}_{K,2} = \big(\sin \frac{3\pi}{4}, \cos \frac{3\pi}{4}\big)^T.
\] Since $\frac{1}{\lambda_{2}}=10^8$, we see that the dominant term in the sum \eqref{eq:Peclet1_2D_exp} will come from $\mathbf{u}_{K,2}$. It can be computed that $\V_\sigma \bigcdot \mathbf{u}_{K,2}<0$. The sign of the P\'eclet number is then determined by the sign of $\mathbf{n}_{K,\sigma} \bigcdot \mathbf{u}_{K,2}$. For Cartesian meshes, the outward normal vectors along the east and north edges are given by $\mathbf{e}_x$ and $\mathbf{e}_y$, respectively. This gives $(\Lam_K^{-1}\V_{\sigma}) \bigcdot \mathbf{e}_{x} \approx -5 \times 10^{7}$ and $(\Lam_K^{-1}\V_{\sigma}) \bigcdot \mathbf{e}_{y} \approx 5 \times 10^{7}$. It is notable here that due to computing the relative strength of advection over diffusion along $\mathbf{u}_{K,1}$ and $\mathbf{u}_{K,2}$, the P\'eclet number along the $x$-direction has an incorrect sign. Moreover, \eqref{eq:Peclet1_2D} yields a P\'eclet number with a very large magnitude in both the $x$- and $y$-directions.   

At this stage, we also recall that the scaling factor $\lambda_\sigma$ in \eqref{eq:advFlux} was for the purpose of stabilising the SG method for advection dominated regimes. However, in this case, where diffusion and advection are both moderate, $\lambda_\sigma$ would always be $10^{-8}$. This results in introducing too much numerical diffusion to the scheme.
We will show in the numerical tests in Section \ref{sec:Numtests} that this causes the accuracy of the scheme to degrade to first order. We note however that in some situations, the choice \eqref{eq:Peclet1_2D} for the P\'eclet number and the additional diffusion it introduces enables us to eliminate spurious oscillations, which can also be beneficial.

\subsubsection{Grid-based P\'eclet number}
Another option would involve taking 
\begin{equation}\label{eq:Peclet2_2D}
P_{K,\sigma} = |\x_K-\x_L| \frac{ \V_{\sigma} \bigcdot \mathbf{n}_{K,\sigma}}{\min(\mathbf{n}_{K,\sigma}^T \Lam_K \mathbf{n}_{K,\sigma}, \mathbf{n}_{L,\sigma}^T \Lam_L \mathbf{n}_{L,\sigma})}.
\end{equation}
The rationale behind the usage of $\mathbf{n}_{K,\sigma}^T \Lam_K \mathbf{n}_{K,\sigma}$ is to directly compute a P\'eclet number that is oriented towards $\mathbf{n}_{K,\sigma}$. In particular, we now see that the numerator $\mathbf{\V_\sigma}\bigcdot\mathbf{n}_{K,\sigma}$ is the strength of advection along $\mathbf{n}_{K,\sigma}$, and the denominator $\mathbf{n}_{K,\sigma}^T \Lam_K \mathbf{n}_{K,\sigma}$ is the strength of diffusion along $\mathbf{n}_{K,\sigma}$. Hence, the P\'eclet number along $\mathbf{n}_{K,\sigma}$ is computed directly without having to go through the basis vectors $\{\mathbf{u}_{K,i}\}_{i=1,\dots,d}$. Moreover, upon writing an orthogonal diagonalisation $\Lam_K = \mathbf{U}_K\mathbf{D}_K\mathbf{U}_K^T$, and denoting by  $\mathbf{B}_K$ the $d\times 1$ vector with $i$th entry equal to $\beta_i$ (see \eqref{eq:onb}), we obtain 
\begin{equation}\nonumber
\begin{aligned} 
\mathbf{n}_{K,\sigma}^T \Lam_K \mathbf{n}_{K,\sigma} &=  \mathbf{n}_{K,\sigma}^T \mathbf{U}_K\mathbf{D}_K\mathbf{U}_K^T \mathbf{n}_{K,\sigma} \\
&= \mathbf{B}_K^T \mathbf{D}_K \mathbf{B}_K \quad \mathrm{ from \,\eqref{eq:onb}}\\
&= \sum_{i=1}^d \beta_i^2 \lambda_{K,i}.
\end{aligned}
\end{equation}
Together with expression \eqref{eq:onb}, we find that the P\'eclet number \eqref{eq:Peclet2_2D} can be written as
\[
P_{K,\sigma} = |\x_K-\x_L| \dfrac{\sum_{i=1}^d \gamma_i\beta_i}{\sum_{i=1}^d \beta_i^2 \lambda_{K,i}}.
\]Here, we see that the quantity in the denominator is always positive; hence it does not affect the sign of $\mathbf{\V_\sigma}\bigcdot\mathbf{n}_{K,\sigma}$. Physically, this means that if material is being transported outside (into) cell $K$, which corresponds to $\mathbf{\V_\sigma}\bigcdot\mathbf{n}_{K,\sigma}$ being positive (negative), then the P\'eclet number preserves this property. Moreover, since $\mathbf{n}_{K,\sigma}$ is an outward unit normal vector, we have \[\mathbf{n}_{K,\sigma}^T\mathbf{n}_{K,\sigma}=\sum_{i=1}^d \beta_i^2 = 1.\]
This shows that the diffusion along $\mathbf{n}_{K,\sigma}$ is a weighted average of the eigenvalues of $\Lam_K$. This is a good weighted average in the sense that if $\mathbf{u}_{K,i}$ is almost orthogonal to $\mathbf{n}_{K,\sigma}$, then $\beta_i=\mathbf{n}_{K,\sigma}^T\mathbf{u}_{K,i}\approx 0$, which means that the corresponding eigenvalue $\lambda_{K,i}$ would only have a minimal contribution to the P\'eclet number.

Considering again $\V=(1,2)^T$ and $\Lam$ as in \eqref{eq:Lam_ill}, we have 
\[\frac{ \V_{\sigma} \bigcdot \mathbf{e}_x}{\mathbf{e}_x^T \Lam_K \mathbf{e}_x} \approx 2, \qquad \frac{ \V_{\sigma} \bigcdot \mathbf{e}_y}{\mathbf{e}_y^T \Lam_K \mathbf{e}_y} \approx 4.\] Here, we now see that the P\'eclet numbers both have correct signs and are moderate along both the $x$- and $y$- directions. Moreover, compared to \eqref{eq:Peclet1_2D}, the choice \eqref{eq:Peclet2_2D} captures properly the strength of advection. In particular, for the choice $\V=(1,2)^T$, we see that the P\'eclet number along the $y$-direction is twice as large as that along the $x$-direction. 

Drawing inspiration from the definition \eqref{eq:Peclet2_2D} of the local grid-based P\'eclet number, we redefine the scaling factor $\lambda_\sigma$ in the modified Scharfetter--Gummel flux \eqref{eq:advFlux} to be
\begin{equation}\label{eq:genLambda}
\lambda_\sigma = \begin{cases}
\min(1,\mathbf{n}_{K,\sigma}^T \Lam_K \mathbf{n}_{K,\sigma}, \mathbf{n}_{L,\sigma}^T \Lam_L \mathbf{n}_{L,\sigma})\quad \mbox{ if } \quad  \sigma \in \edgescv \cap \mathcal{E}_L,\\
\min(1,\mathbf{n}_{K,\sigma}^T \Lam_K \mathbf{n}_{K,\sigma}) \quad \mbox{ otherwise. }
\end{cases}
\end{equation} 
These choices mitigate the introduction of excessive numerical diffusion; hence allowing us to preserve the second-order accuracy expected from the complete flux scheme, as will be demonstrated in the numerical tests in Section \ref{sec:Numtests}.

\subsection{Complete flux scheme in higher dimensions} \label{sec:CFnD} 
In this section, we discuss the formulation of the complete flux scheme in higher dimensions, starting with dimension $d=2$. We use the notation $\x = (x,y)$, and for simplicity of exposition, consider Cartesian meshes.  We start by considering an edge $\sigma\in\edgescv$ being shared by cells $K$ and $L$. This is described by $x=x_\sigma$, $y_S<y<y_N$ (Figure \ref{fig.Inhomog_fluxes_cart}, left). We then find two points $\x_K$ and $\x_L$ in $K$ and $L$, respectively, and construct a segment orthogonal to $\sigma$ that passes through these points (in Figure \ref{fig.Inhomog_fluxes_cart}, left, this pertains to the segment that lies on the line $y=y_\sigma$). We then construct rectangular regions $K_\sigma' = (x_K,x_\sigma) \times (y_S,y_N)$  and $L_\sigma' = (x_\sigma, x_L) \times (y_S,y_N)$ associated to cells $K$ and $L$, respectively (Figure \ref{fig.Inhomog_fluxes_cart}, right).
\begin{figure}[h]
	\caption{Regions involved for the inhomogeneous fluxes.}
	\begin{tabular}{cc}
	\includegraphics[width=0.45\linewidth]{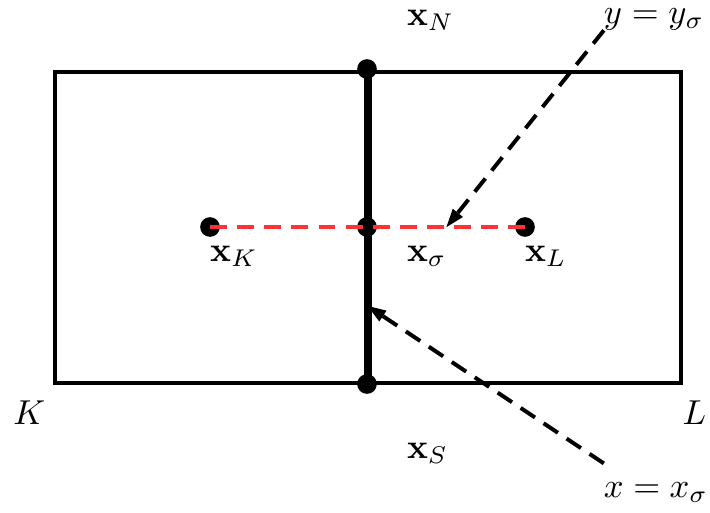} & \includegraphics[width=0.45\linewidth]{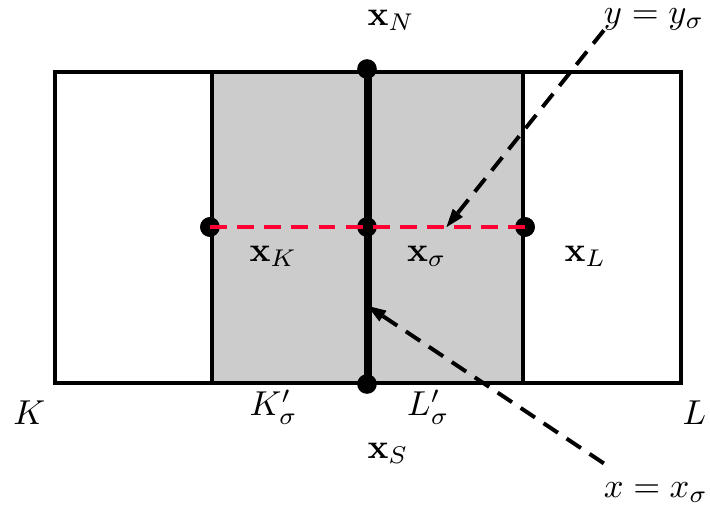} 
\end{tabular}
	\label{fig.Inhomog_fluxes_cart}
\end{figure}

To obtain an approximation of the complete flux $F_{K,\sigma} \approx \int_\sigma (-\Lam \nabla c + c\V) \bigcdot \mathbf{n}_{K,\sigma}\mathrm{d}\ell$ along $\sigma$, we treat the advection-diffusion equation \eqref{eq:adv-diff} as a quasi-one-dimensional boundary value problem. That is, we treat it as an ODE in $x$ by writing  
\begin{equation}\label{eq:quasi1D}
\begin{aligned}
&\frac{\partial}{\partial x} \bigg((-\Lam\nabla c + c\V ) \bigcdot \mathbf{e}_x\bigg) = s - \frac{\partial }{\partial y}\bigg( (-\Lam\nabla c + c\V) \bigcdot \mathbf{e}_y\bigg), \quad x_K < x < x_L, \\
& c(\x_K) = c_K \qquad c(\x_L) = c_L,
\end{aligned}
\end{equation}
where $\mathbf{e}_x = (1,0), \mathbf{e}_y = (0,1)$ are the standard basis vectors in $\R^2$. We note here that for the particular edge under consideration, $\mathbf{n}_{K,\sigma} = \mathbf{e}_x$. To find the value of $(-\Lam\nabla c + c\V ) \bigcdot \mathbf{e}_x$, we then solve the quasi-one-dimensional problem \eqref{eq:quasi1D} in a similar manner as Section \ref{sec:CF1d}.

\begin{remark}If $\sigma$ is a northern or southern edge, a similar approach is used for computing the complete flux, by treating \eqref{eq:adv-diff} as an ODE in $y$.
	\end{remark}

\subsubsection{Inhomogeneous flux: source and cross fluxes}\label{sec:IFnd}
Upon solving the quasi-one-dimensional problem \eqref{eq:quasi1D}, the value of $(-\Lam\nabla c + c\V ) \bigcdot \mathbf{e}_x$ at $x_\sigma$ will of course consist of both a homogeneous and an inhomogeneous component. Since we already have a discretised homogeneous flux \eqref{eq:HF} in Section \ref{sec:HF}, we only focus on the inhomogeneous component. By a generalisation of \eqref{eq:inhomog_flux_1D}, the inhomogeneous component of $(-\Lam\nabla c + c\V ) \bigcdot \mathbf{e}_x$ at $x=x_\sigma$ is given by
\begin{equation}\label{eq:partial_inhomog_flux}
|\x_K-\x_L|
\bigg(Z\big(-P_{K,\sigma},1-\alpha_{K,\sigma} \big)\hat{s}_{K,\sigma} -Z\big(P_{K,\sigma},\alpha_{K,\sigma} \big)\hat{s}_{L,\sigma}\bigg),
\end{equation}
where,  similar to \eqref{eq:sigma_lin_comb_1D}, $\alpha_{K,\sigma} \in (0,1)$ is a scaled coordinate such that
\begin{equation}\label{eq:sigma_lin_comb_2D}\x_\sigma = (1-\alpha_{K,\sigma}) \x_K + \alpha_{K,\sigma} \x_L,\end{equation}
 $\hat{s}_{K,\sigma}$ and $\hat{s}_{L,\sigma}$ are the average values of the right hand side of \eqref{eq:quasi1D} along $x_K<x<x_\sigma$ and $x_\sigma<x<x_L$, respectively, i.e.
\[
\hat{s}_{K,\sigma} = \frac{1}{|\x_K-\x_\sigma|}\int_{x_K}^{x_\sigma} \bigg(s - \frac{\partial }{\partial y}\bigg( \big(-\Lam\nabla c + c\V\big) \bigcdot \mathbf{e}_y\bigg) \bigg)\mathrm{d}x,     \]
\[
\hat{s}_{L,\sigma} = \frac{1}{|\x_L-\x_\sigma|}\int_{x_\sigma}^{x_L} \bigg(s - \frac{\partial }{\partial y}\bigg( \big(-\Lam\nabla c + c\V\big) \bigcdot \mathbf{e}_y\bigg) \bigg)\mathrm{d}x.   
\] 
Substituting the above expressions for $\hat{s}_{K,\sigma}, \hat{s}_{L,\sigma}$ into \eqref{eq:partial_inhomog_flux} and taking the integral over $\sigma$, we obtain an approximation to the inhomogeneous component $F_{K,\sigma}^I$ of the flux $\int_\sigma (-\Lam \nabla c + c\V) \bigcdot \mathbf{n}_{K,\sigma} \, \mathrm{d}\ell$. In particular, the inhomogeneous flux $F_{K,\sigma}^I$ is given by:

\begin{equation} \label{eq:inhomog_fluxes} 
 F_{K,\sigma}^I =  |\x_K-\x_L||y_N-y_S|\bigg(Z(-P_{K,\sigma},1-\alpha_{K,\sigma} )\tilde{s}_{K,\sigma} -Z(P_{K,\sigma},\alpha_{K,\sigma} )\tilde{s}_{L,\sigma}\bigg),
\end{equation}
where
\[
\tilde{s}_{K,\sigma} = \frac{1}{|y_N-y_S|}\int_{y_S}^{y_N} \hat{s}_{K,\sigma}\, \mathrm{d}y,
\]
\[
\tilde{s}_{L,\sigma} = \frac{1}{|y_N-y_S|} \int_{y_S}^{y_N} \hat{s}_{L,\sigma}\, \mathrm{d}y.
\]
\begin{remark}[Extension into 3D] \label{rem:3Dext}
	The formulation of the inhomogeneous flux $F_{K,\sigma}^I$ by
	taking the integral of \eqref{eq:partial_inhomog_flux} over $\sigma$ can straightforwardly be applied to obtain inhomogeneous fluxes in 3D. The only modification to \eqref{eq:inhomog_fluxes} would be the definition of $\tilde{s}_{K,\sigma}$ and $\tilde{s}_{L,\sigma}$, due to an additional term that would come from the partial derivative with respect to $z$ in \eqref{eq:quasi1D}.
\end{remark}

We now decompose the inhomogeneous flux into two components by writing
	\[
	F_{K,\sigma}^{I} = F_{K,\sigma}^{I,s} - F_{K,\sigma}^{I,c},
	\] where 
\begin{equation}\nonumber
\begin{aligned}
F_{K,\sigma}^{I,s} =& \frac{|\x_K-\x_L|}{|\x_K-\x_\sigma|}Z(-P_{K,\sigma},1-\alpha_{K,\sigma} ) \int_{y_S}^{y_N}\int_{x_K}^{x_\sigma} s \,\mathrm{d}x\,\mathrm{d}y
\\
&\!\!- \frac{|\x_K-\x_L|}{|\x_L-\x_\sigma|}Z(P_{K,\sigma},\alpha_{K,\sigma} ) \int_{y_S}^{y_N}\int_{x_\sigma}^{x_L} s \,\mathrm{d}x\,\mathrm{d}y.
\end{aligned}
\end{equation}
Under the assumption that $s$ is piecewise constant with values $s_K$ and $s_L$ on cells $K$ and $L$, respectively, the expression for $F_{K,\sigma}^{I,s}$ simplifies to
\begin{equation}\label{eq:source_flux}
F_{K,\sigma}^{I,s} = \frac{|\x_K-\x_L|}{|\x_K-\x_\sigma|}Z(-P_{K,\sigma},1-\alpha_{K,\sigma} ) |K_\sigma'| s_K -\frac{|\x_K-\x_L|}{|\x_L-\x_\sigma|}Z(P_{K,\sigma},\alpha_{K,\sigma} ) |L_\sigma'|s_L,
\end{equation}
where $K_\sigma'$ and $L_\sigma'$ are rectangular regions contained in $K$ and $L$, respectively (see Figure \ref{fig.Inhomog_fluxes_cart}). Since $F_{K,\sigma}^{I,s}$ involves the source terms, we call it the \emph{source flux}.
On the other hand, we have
\begin{equation}\nonumber
	\begin{aligned}
	F_{K,\sigma}^{I,c} =& \frac{|\x_K-\x_L|}{|\x_K-\x_\sigma|}Z(-P_{K,\sigma},1-\alpha_{K,\sigma} )\int_{y_S}^{y_N}\int_{x_K}^{x_\sigma} \frac{\partial }{\partial y}\bigg( (-\Lam\nabla c + c\V) \bigcdot \mathbf{e}_y\bigg)\mathrm{d}x\,\mathrm{d}y
	\\ &\,-\frac{|\x_K-\x_L|}{|\x_L-\x_\sigma|}Z(P_{K,\sigma},\alpha_{K,\sigma} )\int_{y_S}^{y_N}\int_{x_\sigma}^{x_L} \frac{\partial }{\partial y}\bigg( (-\Lam\nabla c + c\V) \bigcdot \mathbf{e}_y\bigg)\mathrm{d}x\,\mathrm{d}y,
	\end{aligned}
\end{equation}
which, upon changing the order of integration, can be simplified to
\begin{equation}\label{eq:cross_flux}
	\begin{aligned}
	F_{K,\sigma}^{I,c} =& \frac{|\x_K-\x_L|}{|\x_K-\x_\sigma|}Z(-P_{K,\sigma},1-\alpha_{K,\sigma} )\big(F_{K_\sigma',\sigma_N} + F_{K_\sigma',\sigma_S}\big)
	\\ &\,-\frac{|\x_K-\x_L|}{|\x_L-\x_\sigma|}Z(P_{K,\sigma},\alpha_{K,\sigma} )\big(F_{L_\sigma',\sigma_N} + F_{L_\sigma',\sigma_S}\big),
	\end{aligned}
\end{equation}
where  $F_{K_\sigma',\sigma_N}, F_{K_\sigma',\sigma_S}$ are the fluxes along the northern and southern edges, respectively, of  $K_\sigma'$, which need to be determined. Similarly, $F_{L_\sigma',\sigma_N}, F_{L_\sigma',\sigma_S}$ are the fluxes along the northern and southern edges, respectively, of  $L_\sigma'$ (see Figure \ref{fig.Inhomog_fluxes_K}).

\begin{figure}[h]
	\centering
	\caption{Cross-fluxes of $\sigma$.}
		\includegraphics[width=0.6\linewidth]{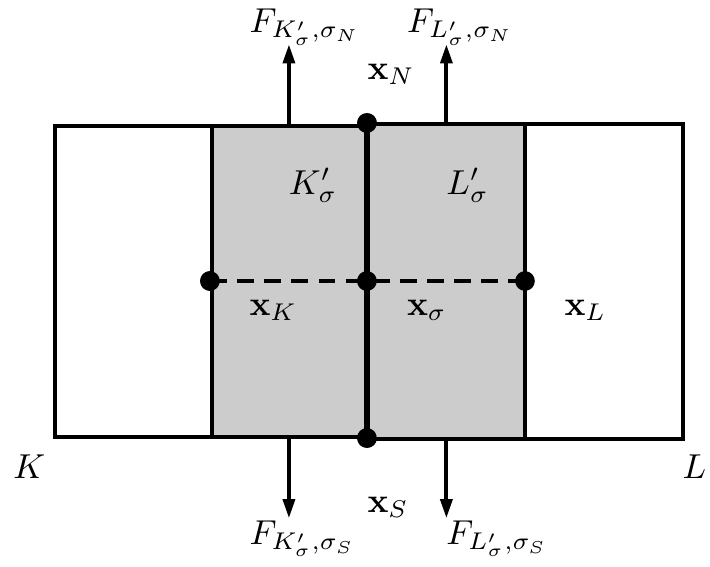} 
	\label{fig.Inhomog_fluxes_K}
\end{figure}
 Since the outward normal vectors at $\sigma_N$ and $\sigma_S$ are both orthogonal to the normal vector at $\sigma$, we call $F_{K,\sigma}^{I,c}$ the \emph{cross-flux}. At this stage, we note that as with the one-dimensional inhomogeneous flux \eqref{eq:inhomog_flux_1D}, the inhomogeneous flux $F_{K,\sigma}^I$ is a higher-order correction term with order $O(|\sigma| |\x_K-\x_L|)$. Hence, a first-order approximation of the cross-flux would be sufficient. This can be done by setting, for Cartesian meshes,
\begin{equation} \label{eq:crossflux_Cart}
F_{K'_\sigma,\sigma_N} = \frac{|\x_K-\x_\sigma|}{|\sigma_N|} F_{K,\sigma_N}^H.
	\end{equation}
	A similar approximation is then used for $F_{K'_\sigma,\sigma_S}$. 
 This completes the definition of the inhomogeneous flux $F_{K,\sigma}^I$ along the interior edge $\sigma\in\edgescv$. A similar process is used to obtain the inhomogeneous fluxes along the other edges of $K$.

To summarise, the inhomogeneous flux $F_{K,\sigma}^I$ along the edge $\sigma$ of cell $K$ is given by 
\begin{equation}\label{eq:inhomo_flux_complete}
\begin{aligned} 
F_{K,\sigma}^I &= F_{K,\sigma}^{I,s} - F_{K,\sigma}^{I,c}\\
&=\dfrac{|\x_K-\x_L|}{|\x_K-\x_\sigma|} \bigg( |K_\sigma'| s_K - F_{K_\sigma',\sigma_N} - F_{K_\sigma',\sigma_S} \bigg) Z(-P_{K,\sigma}, 1-\alpha_{K,\sigma}) \\
& \quad - \dfrac{|\x_K-\x_L|}{|\x_L-\x_\sigma|}\bigg( |L_\sigma'| s_L - F_{L_\sigma',\sigma_N} - F_{L_\sigma',\sigma_S} \bigg) Z(P_{K,\sigma}, \alpha_{K,\sigma}).
\end{aligned}
\end{equation}
On square or rectangular meshes, we can approximate the cross-fluxes in terms of the homogeneous fluxes by using  \eqref{eq:crossflux_Cart} to obtain
\begin{equation}\nonumber
\begin{aligned} 
F_{K,\sigma}^I &= \dfrac{|\x_K-\x_L|}{|\x_K-\x_\sigma|} \bigg( |K_\sigma'| s_K - \frac{|\x_K-\x_\sigma|}{|\sigma_N|}F_{K,\sigma_N}^H - \frac{|\x_K-\x_\sigma|}{|\sigma_S|}F_{K,\sigma_S}^H \bigg) Z(-P_{K,\sigma}, 1-\alpha_{K,\sigma}) \\
& \quad - \dfrac{|\x_K-\x_L|}{|\x_L-\x_\sigma|}\bigg( |L_\sigma'| s_L - \frac{|\x_L-\x_\sigma|}{|\sigma_N|}F_{L,\sigma_N}^H - \frac{|\x_L-\x_\sigma|}{|\sigma_S|}F_{L,\sigma_S}^H \bigg) Z(P_{K,\sigma}, \alpha_{K,\sigma}).
\end{aligned}
\end{equation}

\begin{lemma}[Conservativity of the inhomogeneous fluxes] \label{lem:cons_IF} The inhomogeneous fluxes \eqref{eq:inhomog_fluxes} are conservative. That is, if $\sigma$ is an edge shared by cells $K$ and $L$, then 
	\[
	F_{K,\sigma}^I + F_{L,\sigma}^I = 0.
 	\]
\end{lemma}
\begin{proof}
	We start by taking note that  $P_{K,\sigma} = -P_{L,\sigma}$ and $\alpha_{L,\sigma} = 1-\alpha_{K,\sigma}$. Substituting these expressions into \eqref{eq:inhomog_fluxes} then gives 
	\begin{equation}\nonumber
	\begin{aligned}
	F_{K,\sigma}^I &= |\x_K-\x_L||y_N-y_S|\bigg(Z(P_{L,\sigma},\alpha_{L,\sigma}) \tilde{s}_{K,\sigma} - Z(-P_{L,\sigma},1-\alpha_{L,\sigma})\tilde{s}_{L,\sigma}\bigg)\\
	&= -F_{L,\sigma}^I.
 	\end{aligned}
	\end{equation}
\end{proof}

	\begin{algorithm}
	\caption{Computation of inhomogeneous fluxes.}\label{algo:inhomogeneous_flux}
	\begin{algorithmic}[1]
		\For {$K\in\mesh$}
		\For {$\sigma\in\edgescv \cap \edges_{\mathrm{int}}$}
		\State Find the corresponding neighbor $L$ that shares $\sigma$ with $K$.
		\State Find points $\x_K, \x_L$ in $K$ and $L$ so that the segment passing through $\x_K, \x_L$ is orthogonal to $\sigma$.
				\State Compute the local grid-based P\'eclet number as in \eqref{eq:Peclet2_2D}.
		\State Construct a rectangular region $K'_\sigma$ with length $|\sigma|$, and width $|\x_K-\x_\sigma|$. 
		\State Similarly, construct a rectangular region $L'_\sigma$.
		\State Obtain an approximation for the value of the source flux $F_{K,\sigma}^{I,s}$ \eqref{eq:source_flux}.
		\State Obtain an approximation for the cross-flux $F_{K,\sigma}^{I,c}$ \eqref{eq:cross_flux}.
		\State Substitute the obtained quantities into \eqref{eq:inhomo_flux_complete} to obtain $F_{K,\sigma}^I$.
		\EndFor
		\EndFor
	\end{algorithmic}
\end{algorithm}

 Algorithm \ref{algo:inhomogeneous_flux} presents a short summary of how to compute the inhomogeneous fluxes.
Although the expression for the inhomogeneous fluxes $F_{K,\sigma}^I$ were derived on two-dimensional Cartesian meshes, Algorithm \ref{algo:inhomogeneous_flux} is also applicable for three-dimensional Cartesian meshes (cf. Remark \ref{rem:3Dext}). 

Finally, we present some details about the implementation of the complete flux scheme. We start by assembling the linear system of equations for the balance and conservation of homogeneous fluxes as in Section \ref{subsec:HF}. Following this, we then add, for each $K\in\mesh,\sigma\in\edgescv$, the inhomogeneous source fluxes $F_{K,\sigma}^{I,s}$ \eqref{eq:source_flux} and cross-fluxes $F_{K,\sigma}^{I,c}$ \eqref{eq:cross_flux} to the flux balance equations in order to obtain \eqref{eq:complete_fluxbal}. This then constitutes the linear system of equations needed for the complete flux scheme. We note here that no modification was made for the equations imposing the conservation of fluxes due to the conservativity of the inhomogeneous fluxes cf. Lemma \ref{lem:cons_IF}.

\section{Numerical tests} \label{sec:Numtests}

In this section, we present some numerical tests to demonstrate the second-order accuracy of the generalised complete flux scheme for the advection-diffusion equation \eqref{eq:adv-diff}. These will be performed on Cartesian meshes with square cells over the domain $\O =(0,1) \times (0,1)$. In particular, for the tests presented below, the homogeneous fluxes will be computed via HMM  \eqref{eq:diff_flux} for diffusion, and SG \eqref{eq:advFlux}  with $\lambda_\sigma$ in \eqref{eq:genLambda}  for advection.  
\subsection{Convergence tests}  \label{sec:simpleTest}
We start with three test cases, where the velocity field $\V$ and the diffusion tensor $\Lam$ are constant over the domain $\O$. For these tests, we solve \eqref{eq:adv-diff} with a given exact solution $c(x,y)=\sin(\pi x) \sin(\pi y)$. Denoting by $\pi_\disc c$ the piecewise constant function reconstructed from the discrete unknowns, we measure the relative errors in the $L^1$-norm 
\[E_1:= \dfrac{\norm{\pi_\disc c - c}{L^1(\O)}}{\norm{c}{L^1(\O)}}.\]
 We note here that the errors are measured in the $L^1$-norm to compare with the tests in \cite{AB11-FVCF}. Similar results are observed upon measuring the errors in the $L^2$-norm. Here, we fix $\V=(1,2)^T$, and take different $\Lam$, corresponding to strong or mild anisotropy, and also to whether the problem is advection dominated or not. Here, the strength of anisotropy is quantified by the condition number $\kappa(\Lam)$ of $\Lam$; we say that anisotropy is strong if $\kappa(\Lam)\geq 10^4$.
\begin{itemize}
	\item Test case 1: constant, homogeneous diffusion, advection dominated. We start with an advection dominated test case, where the diffusion is a scalar, i.e., $\Lam= 10^{-8}\mathbf{I}_d$.

\begin{table}[h!]
	\caption{Relative errors in the solution profile, test case 1.}
	\begin{minipage}{0.4\textwidth}\label{tab:conv1}
		\centering
		\begin{tabular}{|c|c|c|} 
			\hline
			Mesh &    $E_1$ &  Order  \\
			\hline
			$16\times16$     & 2.7601e-02 &  \\
			\hline
			$32\times32$      & 7.2298e-03 & 1.9326\\
			\hline
			$64\times64$     & 1.8437e-03 & 1.9713 \\
			\hline
			$128\times128$     & 4.6542e-04 & 1.9860\\
			\hline
						$256\times256$     & 1.1707e-04 & 1.9911\\
			\hline
		\end{tabular} 
		\subcaption{P\'eclet number \eqref{eq:Peclet1_2D}}
	\end{minipage}
	\hfillx
	\begin{minipage}{0.4\textwidth}
		\centering
		\begin{tabular}{|c|c|c|}
			\hline
			Mesh &    $E_1$ &  Order  \\
\hline
$16\times16$     & 2.7601e-02 &  \\
\hline
$32\times32$      & 7.2298e-03 & 1.9326\\
\hline
$64\times64$     & 1.8437e-03 & 1.9713 \\
\hline
$128\times128$     & 4.6542e-04 & 1.9860\\
\hline
$256\times256$     & 1.1707e-04 & 1.9911\\
\hline
		\end{tabular} 
		\subcaption{P\'eclet number \eqref{eq:Peclet2_2D}}
	\end{minipage}
\end{table}
 
\item Test case 2: strong anisotropy (not aligned with the mesh), moderate advection. Here, advection is said to be moderate in the sense that the grid-based P\'eclet number defined in \eqref{eq:Peclet2_2D} $P_{K,\sigma} <1 $ for all $K\in\mesh,\sigma\in\edges_{K}$. For this test case, we consider the diffusion tensor 

\[\Lam=\dfrac{1}{2} \begin{bmatrix}
1+10^{-8} & 1-10^{-8}  \\
1-10^{-8} & 1+10^{-8}
\end{bmatrix}.\] The purpose of this choice for the diffusion tensor is twofold: firstly, it shows that the grid-based P\'eclet number \eqref{eq:Peclet2_2D} is better than the eigenvector-based P\'eclet number \eqref{eq:Peclet1_2D}. Secondly, this shows that the combination of the HMM and CF yields a second-order scheme for problems with strong anisotropy (even if the anisotropy is not aligned with the mesh). In this case, the condition number of $\Lam$ is $\kappa(\Lam)=10^8$.

   \begin{table}[h!]
	\caption{Relative errors in the solution profile, test case 2.}
	\begin{minipage}{0.4\textwidth}\label{tab:conv2}
		\centering
		\begin{tabular}{|c|c|c|} 
			\hline
			Mesh &    $E_1$ &  Order  \\
			\hline
			$16\times16$     & 2.7494e-02 &  \\
			\hline
			$32\times32$      & 1.4338e-02 & 0.9392\\
			\hline
			$64\times64$     & 7.5341e-03 & 0.9283\\
			\hline
			$128\times128$     & 3.8932e-03 & 0.9524\\
			\hline
						$256\times256$     & 1.9833e-03 & 0.9730\\
			\hline
		\end{tabular} 
		\subcaption{P\'eclet number \eqref{eq:Peclet1_2D}}
	\end{minipage}
	\hfillx
	\begin{minipage}{0.4\textwidth}
		\centering
		\begin{tabular}{|c|c|c|}
			\hline
			Mesh &    $E_1$ &  Order   \\
			\hline
			$16\times16$     & 1.1273e-02 &  \\
			\hline
			$32\times32$      & 2.8457e-03 & 1.9859\\
			\hline
			$64\times64$     & 7.1305e-04 & 1.9967 \\
			\hline
			$128\times128$     & 1.7835e-04 & 1.9992 \\
			\hline
						$256\times256$     & 4.4592e-05 & 1.9998\\
			\hline
		\end{tabular} 
		\subcaption{P\'eclet number \eqref{eq:Peclet2_2D}}
	\end{minipage}
\end{table}
\item Test case 3: strong anisotropy (almost aligned with the mesh), advection-dominated. For our third test, we consider a diffusion tensor which features strong anisotropy that is almost aligned with the mesh, where diffusion is very weak in the $y$-direction, but moderate in the $x$-direction. This is done by taking
\[
\Lam=\begin{bmatrix}
1.5 & 10^{-4}  \\
10^{-4} & 10^{-8}\end{bmatrix}.
\]
Here, the condition number is $\kappa(\Lam)\approx 4.5\times 10^8$, and the eigenvalues are given by $\lambda_{1}\approx 3.33\times 10^{-9}, \lambda_2 \approx 1.5$, with corresponding eigenvectors
\[\mathbf{u}_1 \approx (-6.66\times 10^{-5}, 1)^T, \quad \mathbf{u}_2 \approx (1,-6.66\times 10^{-5})^T. \] 
This aims to show that even for a strongly anisotropic diffusion tensor and a relatively strong advection, by making a proper choice for the P\'eclet number, the combined HMM--CF method still has second-order accuracy.  

\begin{table}[h]
	\caption{Relative errors in the solution profile, test case 3.}
\begin{minipage}{0.4\textwidth}\label{tab:conv3}
	\centering
	\begin{tabular}{|c|c|c|} 
		\hline
		Mesh &    $E_1$ &  Order  \\
		\hline
		$16\times16$     & 3.3812e-02 &  \\
		\hline
		$32\times32$      & 1.4982e-02 & 1.1743\\
		\hline
		$64\times64$     & 6.8786e-03 & 1.1230 \\
		\hline
		$128\times128$     & 3.2523e-03 & 1.0806\\
		\hline
				$256\times256$     & 1.5617e-03 & 1.0583\\
		\hline
	\end{tabular} 
	\subcaption{P\'eclet number \eqref{eq:Peclet1_2D}}
\end{minipage}
\hfillx
\begin{minipage}{0.4\textwidth}
	\centering
	\begin{tabular}{|c|c|c|}
		\hline
		Mesh &    $E_1$ &  Order   \\
		\hline
		$16\times16$     & 8.3214e-03 &  \\
		\hline
		$32\times32$      & 2.4269e-03 & 1.7776\\
		\hline
		$64\times64$     & 6.6236e-04 & 1.8734 \\
		\hline
		$128\times128$     & 1.7369e-04 & 1.9310\\
		\hline
				$256\times256$     & 4.4586e-05 & 1.9618\\
		\hline
	\end{tabular} 
	\subcaption{P\'eclet number \eqref{eq:Peclet2_2D}}
\end{minipage}
\end{table}
\end{itemize}
For the first test case, we note that the P\'eclet numbers \eqref{eq:Peclet1_2D} and \eqref{eq:Peclet2_2D} are identical. In particular, since the diffusion is a scalar, the generalised complete flux scheme returns to the formulation presented in \cite{AB11-FVCF}, which we expect to achieve second-order accuracy. This is confirmed by the numerical tests presented in Table \ref{tab:conv1}.  
Now, we see in Tables \ref{tab:conv2} and \ref{tab:conv3} that the grid-based P\'eclet number \eqref{eq:Peclet2_2D} maintains the second-order accuracy of the complete flux scheme. On the other hand, using the eigenvector-based P\'eclet number \eqref{eq:Peclet1_2D} reduces the complete flux scheme into a first-order scheme. This is due to the fact that \eqref{eq:Peclet1_2D} causes the numerical scheme to decide whether the problem is advection dominated based on the eigenvalues of $\Lam$ (i.e., the P\'eclet number is always large when the eigenvalues of $\Lam$ have high contrast), which is due to computing the strength of advection over diffusion along the directions of the eigenvectors of $\Lam$, as discussed in Section \ref{sec:CFnD}. This phenomenon is also illustrated in the third test case, for which diffusion is only weak in the $y$-direction. By using \eqref{eq:Peclet1_2D}, the P\'eclet number has a magnitude of $3.99\times 10^4$ in the $x$-direction and a magnitude of $5.99\times 10^8$ in the $y$-direction; hence, diffusion is interpreted to be weak in all directions. On the other hand, by using \eqref{eq:Peclet2_2D}, the P\'eclet number has a magnitude of $6.66 \times 10^{-1}$ and $2\times 10^8$ in the $x$- and $y$-directions, respectively, so diffusion is weak only in the $y$-direction. 

Having illustrated the second-order convergence of the scheme, we now proceed to test the limits of the scheme by performing two extreme tests. In these cases, an analytical solution is not available, and hence, we analyse the qualitative aspects of the numerical solutions.
\subsection{Strong anisotropy, heterogeneous, and advection-dominated} \label{sec:hetero_test}
We present a numerical test which involves a strongly heterogeneous and anisotropic diffusion tensor, as described in \cite{VDM11-adv-diff,ESZ09-ADG}. This will be referred to as test case 4. For this test, an exact analytic solution is not available, so we comment on the qualitative properties of the numerical solution. Here, homogeneous Dirichlet boundary conditions are imposed. The diffusion tensor is piecewise constant, defined in the following subdomains: $\O_1=(0,2/3) \times (0,2/3), \O_2 = (2/3,1)\times (0,2/3), \O_3 = (2/3,1) \times (2/3,1), \O_4 = (0,2/3)\times(2/3,1)$, with
\[
\Lam = 
\begin{bmatrix}
10^{-6} & 0 \\
0 & 1
\end{bmatrix} \qquad \mbox{ in } \O_1 \mbox{ and } \O_3,
\]
and
\[\Lam = 
\begin{bmatrix}
1 & 0 \\
0 & 10^{-6}
\end{bmatrix} \qquad \mbox{ in } \O_2 \mbox{ and } \O_4.
\]
The velocity field considered is $\V=(40x(2y-1)(x-1),-40y(2x-1)(y-1))^T$, which simulates a counterclockwise rotation (see Figure \ref{fig.param}). 
\begin{figure}
	\caption{Data for test case 4 (left: velocity field; right: diffusion tensor).}
	\begin{tabular}{cc}
		\includegraphics[width=0.5\linewidth]{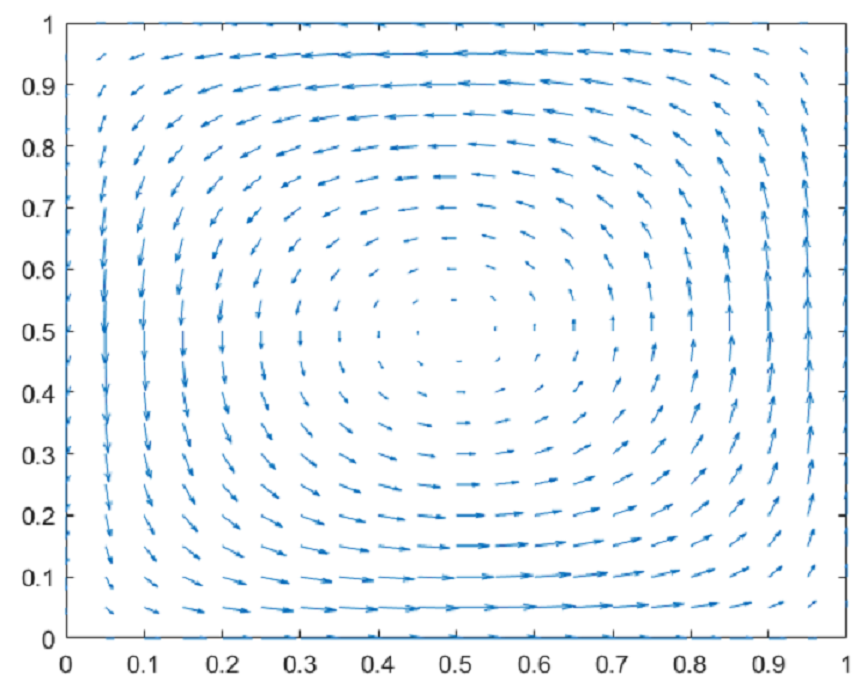} &\qquad \includegraphics[width=0.4\linewidth]{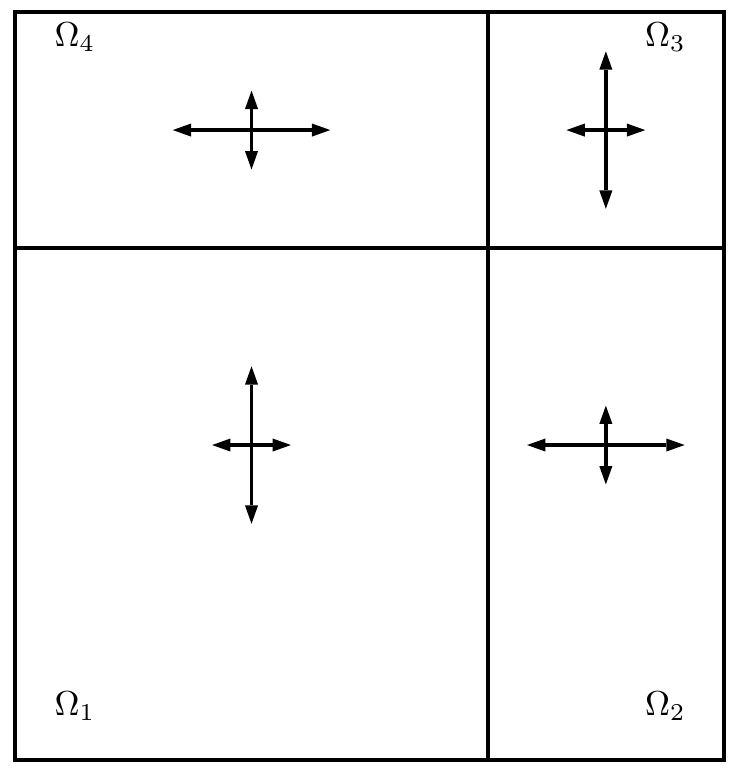} 
	\end{tabular}
	\label{fig.param}
\end{figure}
The source term is a ring positioned at a distance of 0.35 from the center of the domain, i.e., $s(x,y)= 10^{-2} \exp(-(r-0.35)^2/0.005)$, where $r^2 = (x-0.5)^2+(y-0.5)^2$ (see Figure \ref{fig.source}). 

\begin{figure}
	\caption{Source term for test case 4.}
	\begin{center}
		\includegraphics[width=0.5\linewidth]{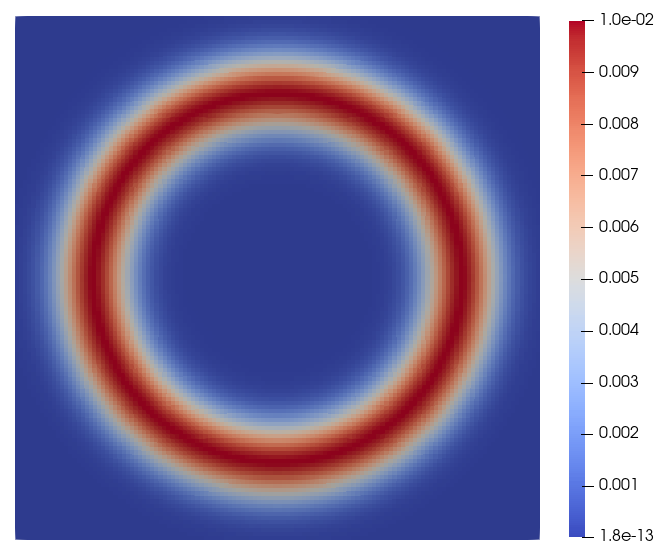} 
	\end{center}
	\label{fig.source}
\end{figure}

We now observe that the source is evenly distributed, and the velocity field is oriented along the direction of increasing diffusivity. Consider now the interface shared by $\O_4$ and $\O_1$. The velocity field causes the solution profile to carry the source in $\O_4$ towards the interface it shares with $\O_1$. Due to the low diffusivity along the $y$-direction in $\O_4$, we expect that the value of the solution to be small near the interface. On the other hand, due to the strong diffusivity along the $y$-direction in $\O_1$, we expect the solution to be large near the interface it shares with $\O_4$. Hence, we expect the exact solution to form internal layers near the interfaces that separate the subdomains. This can be observed in the solution profiles presented in Figure \ref{fig.hetero_test_case}. A better visualisation is given by 3D plots in Figure \ref{fig.hetero_test_case_3D}. Moreover, we observe that the numerical solution does not yield non-physical negative values (minimum value is non-negative). We also note that in the eyeball norm, the numerical solution at the coarse mesh with $60\times 60$ cells is already very close to what we obtain on the very fine mesh of $480\times 480$ cells, which is expected from the complete flux scheme. Also, the maximum value of $7.3\times 10^{-4}$ observed here is also very close to the maximum value ranging from $6.7\times 10^{-4}-6.9 \times 10^{-4}$ observed in \cite{VDM11-adv-diff,ESZ09-ADG}. 
\begin{figure}
	\caption{Solution profile, test case 4, velocity field with counterclockwise orientation (left: numerical solution on $60\times60$ cells, right: numerical solution on $480\times480$ cells).}
	\begin{tabular}{cc}
		\includegraphics[width=0.45\linewidth]{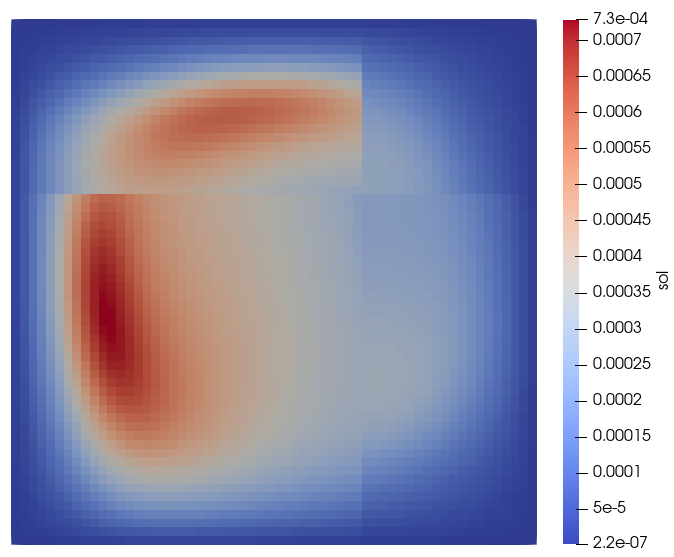} & \includegraphics[width=0.45\linewidth]{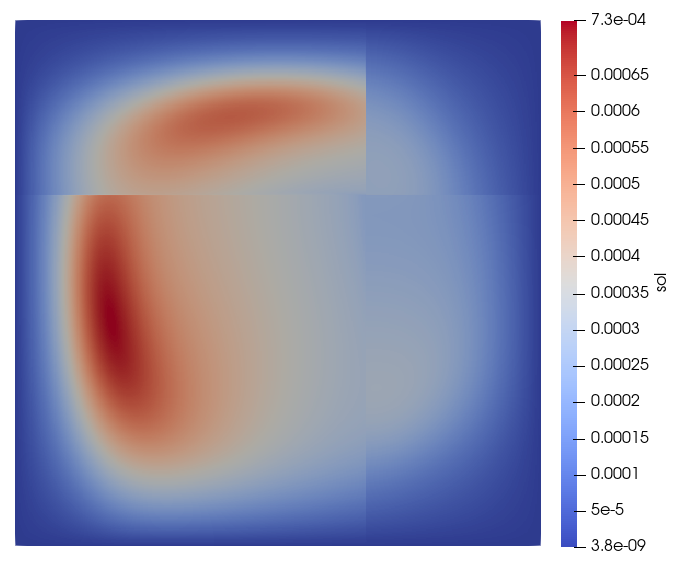} 
	\end{tabular}
	\label{fig.hetero_test_case}
\end{figure}
\begin{figure}
	\caption{Solution profile, test case 4, velocity field with clockwise orientation (left: numerical solution on $60\times60$ cells, right: numerical solution on $480\times480$ cells).}
	\begin{tabular}{cc}
		\includegraphics[width=0.45\linewidth]{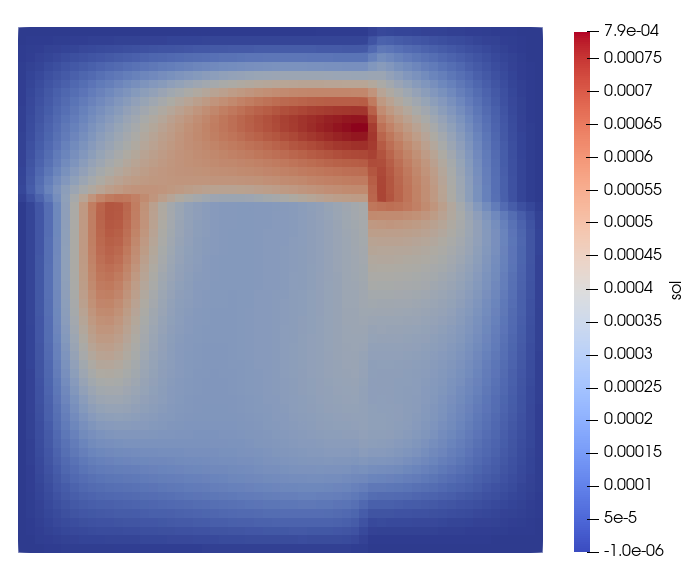} & \includegraphics[width=0.45\linewidth]{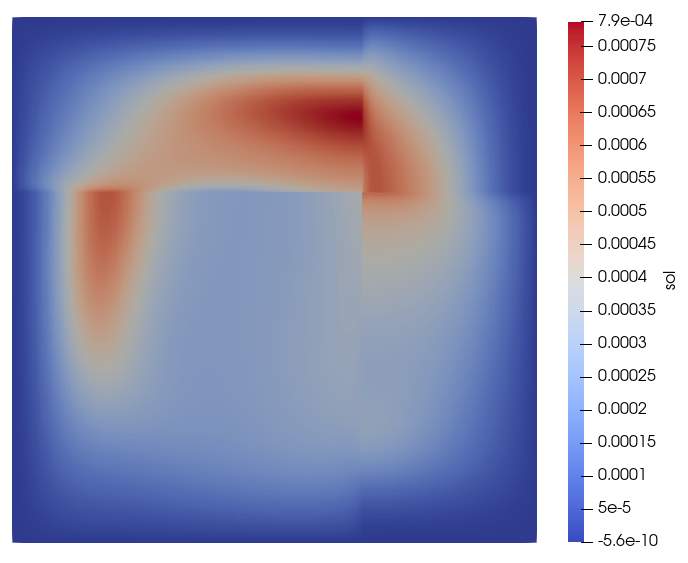} 
	\end{tabular}
	\label{fig.hetero_test_case1}
\end{figure}

Another point of comparison can be made by reversing the flow, that is, by taking $-\V$, whilst retaining all the other parameters used for the test case. In this case, the velocity field is now oriented along the direction of decreasing diffusivity; hence, we expect the solution to be continuous near the interfaces. This behavior can be seen in Figure \ref{fig.hetero_test_case1} and Figure \ref{fig.hetero_test_case_3D}, right. As with the initial test, we observe that the numerical solution at the coarse mesh with $60\times 60$ cells is already very close to what we obtain on the very fine mesh of $480\times 480$ cells. Also, the maximum value of $7.9\times 10^{-4}$ observed here is also very close to the maximum value of $7.3\times 10^{-4}$ observed in \cite{ESZ09-ADG}. We note however that negative values are present in the numerical solutions, but the effect is not very significant, since they are only very small in magnitude: $10^{-6}$ on the coarse mesh, and $10^{-10}$ on the fine mesh.

\begin{figure}
	\caption{Solution profile on a coarse grid with $60 \times 60$ cells, test case 4, (left: velocity field with counterclockwise orientation, right: velocity field with clockwise orientation).}
	\begin{tabular}{cc}
		\includegraphics[width=0.45\linewidth]{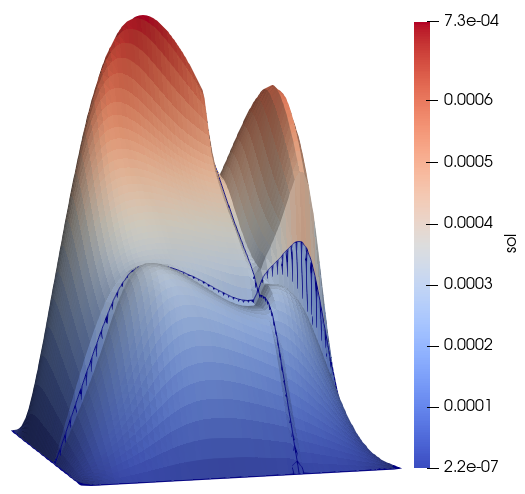} & \includegraphics[width=0.45\linewidth]{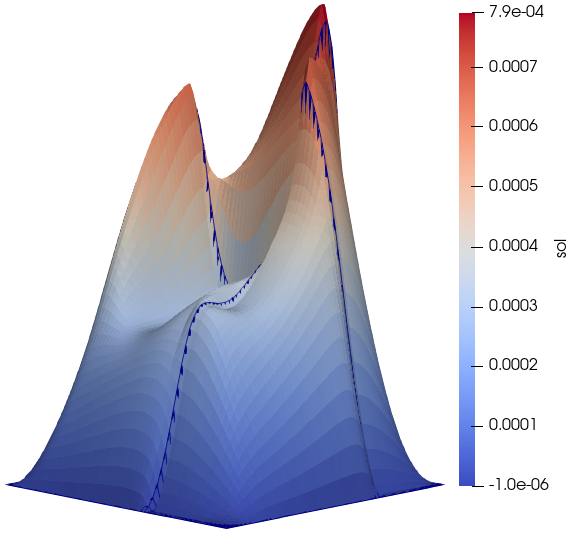} 
	\end{tabular}
	\label{fig.hetero_test_case_3D}
\end{figure}
\subsection{Monotonicity test} \label{sec:monTest}
Finally, we present a test from \cite{L10-monotoneFV}, which we refer to as test case 5. This is an extreme test in the sense that many linear methods result in a significant violation of the discrete maximum principle, and consequently produce numerical solutions that exhibit non-physical oscillations. We consider $\O=(0,1)^2 \setminus [4/9,5/9]^2$; that is, the domain is a square with a hollow square interior, resulting in a boundary consisting of two disjoint parts $\Gamma_1$ and $\Gamma_2$ (see Figure \ref{fig.mon.data}).  Here, the source term is set to be $s=0$, and Dirichlet boundary conditions are imposed, with $c=0$ on $\Gamma_1$ and $c=2$ on $\Gamma_2$. 
\begin{figure}
	\caption{Domain and data for test case 5.}
	\begin{center}
		\includegraphics[width=0.35\linewidth]{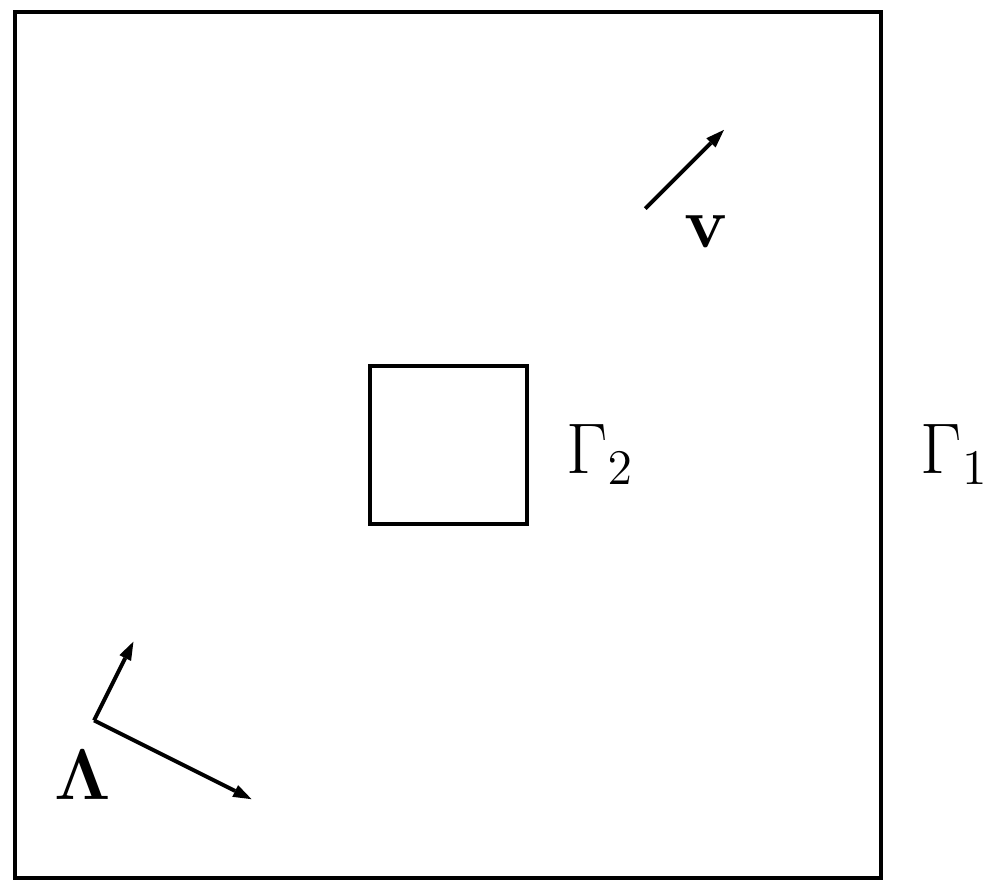} 
	\end{center}
	\label{fig.mon.data}
\end{figure}

\noindent The velocity field is given by $\V=(700,700)^T$, and the diffusion tensor reads
\[
\Lam = R(-\pi/6)\begin{bmatrix}
1000 & 0 \\
0 & 1
\end{bmatrix} R(\pi/6), \qquad R(\theta) = \begin{bmatrix}
\cos \theta & \sin \theta \\
-\sin\theta & \cos\theta
\end{bmatrix}.
\]

\begin{figure}
	\caption{Solution profile, test case 5 (left: numerical solution on $45\times45$ cells, right: numerical solution on $360\times360$ cells).}
	\begin{tabular}{cc}
		\includegraphics[width=0.45\linewidth]{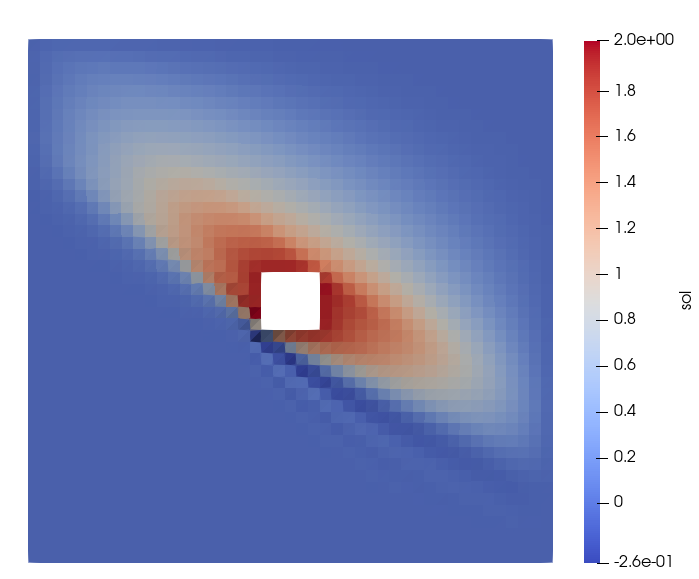} & \includegraphics[width=0.45\linewidth]{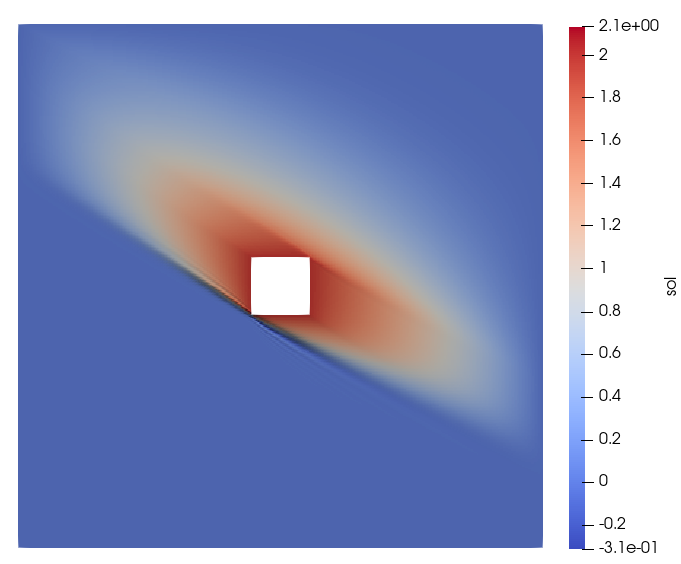} 
	\end{tabular}
	\label{fig.mon_test_case}
\end{figure}
For this test case, the exact solution is expected to be bounded by 0 and 2. Here, we see in Figure \ref{fig.mon_test_case} that the maximum value is close to what is expected, with slight overshoots occurring at the top right corner of the boundary $\Gamma_2$. On the other hand, we see that the complete flux scheme does not guarantee the non-negativity of the numerical solution, with undershoots occurring near the lower left corner of $\Gamma_2$. Non-physical oscillations are also detected in these regions. Although we observe these non-physical qualities, it still offers an improvement over some classical methods, e.g. the lowest-order Raviart–Thomas mixed finite element method, for which the numerical solution is negative over almost half of the computational domain \cite{LSY09-FV-diff}. In particular, the overshoots and undershoots for the generalised complete flux scheme are only observed locally near the top right and lower left regions of $\Gamma_2$ and only cover $16\%$ and $6\%$ of the domain for the coarse and fine meshes, respectively. We note however that the negative values are slightly worse on the finer mesh. These negative non-physical oscillations propagate from the lower left corner of $\Gamma_2$. By refining the mesh, the area of each cell is smaller. Hence, the very strong diffusion along the direction of the eigenvector $(\cos\frac{\pi}{6},-\sin\frac{\pi}{6})^T$ has less area to spread the value of $c=2$ properly, leading to larger negative values.

\section{Summary and outlook} \label{sec:Summary}
In this paper, we were able to present a generalised complete flux scheme for anisotropic advection-diffusion problems. The main novelty comes from splitting the diffusive and advective components of the flux, which allows for a combination of different numerical discretisations with the CF method. This resulted in a scheme which can be applied to problems with anisotropic diffusion tensors. Related to this, we introduced a grid-based local P\'eclet number \eqref{eq:Peclet2_2D}, which allows us to capture properly the local strength of advection over diffusion. Such a generalisation of the P\'eclet number, which allows us to deal with anisotropic diffusion not aligned with the mesh is relatively new, see e.g. \cite[Section 3.2.1.3]{HHOBook20}. The discussion provided in Section \ref{sec:Pecletno} of this work gives a detailed presentation as to why \eqref{eq:Peclet2_2D} is a good choice for defining the local P\'eclet number. Another important contribution is an alternative presentation of the CF method in two dimensions, which can straightforwardly be extended into three dimensions. Moreover, it also provides a framework for computing inhomogeneous fluxes on irregular meshes. The only caveat here is that on non-rectangular meshes, the cross-fluxes no longer lie on cell faces. Hence, we need to perform interpolations and a change of coordinates (similar to those described in \cite{ABK17-FVCF-2D}) in order to obtain approximations for the cross-fluxes and source fluxes in lines 8 and 9 of Algorithm \ref{algo:inhomogeneous_flux}. This is not straightforward, and will be the purpose of future research. The numerical tests presented in Section \ref{sec:simpleTest} illustrate the second-order accuracy of the generalised CF method, even for strong anisotropy and advection dominated problems. Moreover, the tests performed in Section \ref{sec:hetero_test} showcase the ability of the generalised CF method to handle strongly anisotropic heterogeneous diffusion tensors. We however notice in Section \ref{sec:monTest} the presence of non-physical oscillations in the numerical solution. This was not unexpected as it has already been remarked in \cite{L10-monotoneFV} that these non-physical oscillations are present in many linear methods, and future work will involve working on how to mitigate these non-physical oscillations. Another interesting aspect for future work will involve the study of non-stationary advection-diffusion problems.
\section{Acknowledgements}
The authors would want to thank Prof. Sorin Pop for the discussions and comments which helped improve the presentation of the generalised local P\'eclet number. We would also like to thank the referees for their detailed feedback, which helped improve the overall presentation of the manuscript.

	\bibliographystyle{abbrv}
	\bibliography{HMM_CF}
\end{document}